\documentclass[11pt]{amsart}
\usepackage{amssymb,amsmath}
\usepackage{bm}
\usepackage[]{graphicx}
\usepackage{color}
\usepackage{amssymb, epsfig}
\numberwithin{equation}{section}

\numberwithin{figure}{section}
\newtheorem{theorem}{Theorem}[section]
\newtheorem{lemma}[theorem]{Lemma}
\newtheorem{definition}[theorem]{Definition}

\newtheorem{remark}[theorem]{Remark}
\begin{document}
\title[The multi-d Stochastic Stefan Financial Model for a portfolio]
{The multi-dimensional Stochastic Stefan\\ Financial Model for a
portfolio of assets}

\author[D.~C. Antonopoulou]{Dimitra Antonopoulou$^{\# *}$}
\author[M. Bitsaki]{Marina Bitsaki$^{\ddag}$}
\author[G. Karali]{Georgia Karali$^{\dag *}$}

\thanks
{$^{\dag}$ Department of Mathematics and Applied Mathematics,
University of Crete, GR--714 09 Heraklion, Greece.}

\thanks
{$^{\ddag}$ Department of Computer Science, University of Crete,
GR--700 13 Heraklion, Greece.}

\thanks
{$^{\#}$ Department of Mathematics, University of Chester,
Thornton Science Park, CH2 4NU, UK}
\thanks
{$^{*}$ Institute of Applied and Computational Mathematics,
FORTH, GR--711 10 Heraklion, Greece.}

%
%
%

%
%

\begin{abstract}
The financial model proposed in this work involves the
liquidation process of a portfolio of $n$ assets through sell or
(and) buy orders placed, in a logarithmic scale, at a (vectorial)
price $x\in\mathbb{R}^n$, with volatility. We present the rigorous
mathematical formulation of this model in a financial setting
resulting to an $n$-dimensional outer parabolic Stefan problem
with noise. The moving boundary encloses the areas of zero
trading, the so-called solid phase. We will focus on a
 case of financial interest when one or more markets are
considered. In particular, our aim is to estimate for a short time
period the areas of zero trading, and their diameter which
approximates the minimum of the $n$ spreads of the portfolio
assets for orders from the $n$ limit order books of each asset
respectively.

 In dimensions $n=3$,
and for zero volatility, this problem stands as a mean field
model for Ostwald ripening, and has been proposed and analyzed by
Niethammer in \cite{nietthes}, and in \cite{JDE12} in a more
general setting. There in, when the initial moving boundary
consists of well separated spheres, a first order approximation
system of odes had been rigorously derived for the dynamics of the
interfaces and the asymptotic profile of the solution. In our
financial case, we propose a spherical moving boundaries approach
where the zero trading area consists of a union of spherical
domains centered at portfolios various prices, while each sphere
may correspond to a different market; the relevant radii
represent the half of the minimum spread. We apply It\^o calculus
and provide second order formal asymptotics for the stochastic
version dynamics, written as a system of stochastic differential
equations for the radii evolution in time. A second order
approximation seems to disconnect the financial model from the
large diffusion assumption for the trading density. Moreover, we
solve the approximating systems numerically.
\end{abstract}
\maketitle \pagestyle{myheadings}
\thispagestyle{plain}
%
%
%
\section{Introduction}
\subsection{A Stefan problem for the liquidation of a portfolio}
Decision making tools play an important role in quantifying the
different sources of uncertainty in portfolio management (such as
prices, market liquidation, etc.), and on deriving efficient
portfolio strategies. Many studies are focused on the portfolio
selection problem where the measuring of the performance of
portfolios is based on various criteria such as the variance of
expected returns, \cite{mark}, risk minimization and utility
maximization, \cite{merton}.

Liquidation of a portfolio of $n$ assets, is the process of
transforming the aforementioned set of assets into cash, for
example through sell and buy orders. A certain question of
significant financial importance that naturally arises concerns
the determination of a profitable price of trading at a specific
time $t$. Moreover, the investor would like to predict an optimal
time for liquidation and the dynamics of the spreads, even for
short time periods.

The proposed model in this paper is applicable to the next
strategy of liquidation summarized as follows. \\
\textit{Sell and buy orders from the limit order book when one or
more markets are considered:}
\begin{enumerate}
\item[--]
We observe the evolution of prices for a portfolio of $n$ assets
of analogous properties traded during the same financial day in
one or different (but interacting) markets, for example
currencies in European Union markets. An asset is defined as
liquid when it is traded through sell or buy orders, and the
prices of zero trading per asset define the relevant spreads. We
aim to estimate (predict) through time an average spread at each
market. The initial data of the problem will be taken from the
limit order books, where the bid and ask prices, as well the
volume of trading are included at discrete times in a financial
day. The bid and ask prices at time $t=0$ will induce the initial
spreads, while the total liquidity will be estimated by using
information from all the interacting markets. \vspace{0.2cm}
\item[--]
Under the assumption of infinitesimally small `tick size', that
is a minimum permitted price increment of the financial market
tending to zero, we may consider continuous price models
(continuous space-like coordinates). A Stefan problem for a Heat
equation with stochastic volatility posed on the liquid phase
will describe the diffusion of the sell or (and) buy orders in
time.
\end{enumerate}

We note that an alternative trading strategy would involve
stop-loss orders, which consist standing orders to sell an asset
when its price drops by a certain percentage. However, this
approach is only temporarily effective, for example during a
breakdown swing of the market, and may not be the optimal one for
a long-term portfolio performance, \cite{zi}. On the other hand,
there exists a behavioral finance characteristic called as `the
disposition effect', which describes the tendency of investors to
sell assets when their price is increasing rather than
decreasing; this effect is difficult to be predicted, see for
example in \cite{he}, for a model of asset liquidation, where the
investors realize utility over gains and losses, or in
\cite{CL,CST} for various financial models estimating the prices
dynamics related to the limit orders market.

\subsection{Motivation for the proposed model}
There exist so far some interesting and rigorous results on
modeling and well posedness of financial Stefan problems for the
Heat equation, even with noise, but up to now are restricted only
in dimension one, where the price of one asset is considered; see
the pioneering works of Ekstr\"om, Zhi Zheng and M\"uller in
\cite{revpap,phd2,main1} for some 2-phases 1-dimensional
stochastic Stefan systems, for sell and buy orders of one asset.
X. Chen and Dai proposed and analyzed an optimal strategy for
multiasset investment on correlated risky assets of a portfolio,
\cite{chenf}, while Altarovici, Muhle-Karbez and Soner in
\cite{amks}, presented an optimal policy and leading order
asymptotics related to multiple risky assets trading with small
and fixed transaction cost.

A natural extension is to consider more than one assets
consisting a portfolio and state analogous Stefan problems in
dimensions $n\geq 2$ with stochastic volatility.

Various deterministic parabolic Stefan problems have been
extensively used for describing the phase separation of alloys
and a relevant mathematical theory is already well established.
See for example the results of Niethammer, X. Chen and Reitich,
Antonopoulou Karali and Yip in \cite{nietthes,XFC1,JDE12}, or for
the quasi-static problem in \cite{abc,af,afk1,afk2,niet2, chen1,
chen2, chen3}. Note that the quasi-static problem approximates the
parabolic one when the diffusion tends to infinity as in the case
of a very large trading activity.

We shall state a stochastic multi-dimensional moving boundary
financial problem, and will enlighten in a financial setting for
our parameters the existing theory for the boundary dynamics
(which is only developed in the deterministic version), mainly
motivated by the work of Niethammer in \cite{nietthes}. We
provide a financial interpretation of the Gibbs Thomson condition
involving the mean curvature of the interface, and propose a
simplified model formulation for the approximation of the initial
moving boundary by spheres; their radii may implement the price
ranges of zero trading around the portfolio intrinsic values. In
this case, the mathematical theory for zero volatility,
\cite{nietthes}, predicts the increasing of the large radii
spheres at the expense of the smaller ones, as for example when
the price strongly surpasses the assets intrinsic value (financial
bubbles).

We use It\^o calculus and derive second order formal asymptotics
for the stochastic dynamics of the moving interface and the
solution of the Heat equation of the Stefan problem. These are
presented as a system of stochastic odes, which is solved
numerically. Our numerical results indicate that in contrast to
the deterministic problem, where static solutions (equilibrium of
one sphere or of many equal spheres) evolve very slowly, when
noise is present initial states of even one sphere may decrease
their volume in relatively small times (or increase). The previous
describes a faster liquidation process (or solidification
process) for non zero volatility.

The investigation of well posedness and regularity for the fully
stochastic version consists a work in progress and it is not
considered in this paper. Moreover, the rigorous mathematical
derivation of the stochastic dynamics for the moving boundary,
which as we shall see involve the mean curvature of the surface,
remains a challenging open problem.

\section{The $n$-dimensional outer stochastic Stefan problem for a portfolio}
\subsection{The mathematical statement}
We consider a \textit{portofolio} of $n\geq 2$ different assets,
and define their trading prices through sell or (and) buy orders
by $x_1$, $x_2$,$\cdots$,$x_n$ respectively in a logarithmic
scale. So, $x:=(x_1,x_2,\cdots,x_n)$ in general belongs to
$\mathbb{R}^n$ (and not restricted as the usual asset prices in
$\mathbb{R}^{+n}$), cf. in \cite{phd2}. Each asset may consist of
one only share, and thus $x_i=x_i(t)$ is the price of the specific
share when traded at time $t$ (enwritten, before the logarithmic
rescaling, in the limit order book of this share). 

Let $w=w(x,t)$ be the \textit{fluctuating density, cf.
\cite{main1}, or volume, cf. \cite{phd2}}, of the portfolio placed
at price $x=(x_1,\cdots,x_n)$. We observe the evolution of the
density $w$ in time, for $t\in[0,T]$ and pose a stochastic heat
equation on a \textit{liquid phase}$:=
\mathbb{R}^n-\mathcal{D}(t)$ with boundary $\Gamma(t)$; this
phase is defined as the complement in $\mathbb{R}^n$ of the areas
of zero trading (i.e., the complement of the solid phase
$\mathcal{D}(t)$).

The liquid phase domain for these $n$ assets describes the set of
prices $x\in\mathbb{R}^n$ on which trading is executed
 and is of course
unknown, while it is one of the portfolio characteristics that we
would like to determine through the moving boundary problem. The
decision of trading, and thus the liquid phase domain, is induced
by the distance of the vector
$$x:=(x_1,x_2,\cdots,x_n),$$ from the boundary $\Gamma=\Gamma(t)$ of
$\mathcal{D}(t)$, and this distance is given by
$${\rm
dist}(x,\Gamma):=\displaystyle{\inf_{y\;on\;\Gamma}}\|x-y\|,$$
where $\Gamma$ is a curve ($n=2$), or surface ($n=3$), or
hyper-surface ($n\geq 4$). Here, $\|\cdot\|$ denotes the
euclidean norm in $\mathbb{R}^n$.

The total price of the portfolio is defined by the euclidean norm
of $x$
$$\|x\|:=\Big{(}x_1^2+x_2^2+\cdots x_n^2\Big{)}^{1/2}.$$
\begin{remark}
The euclidean norm in $\mathbb{R}^n$ was used in the mathematical
analysis of the physical problem of phase separation of alloys
and the asymptotics formulae derived in \cite{nietthes}
(involving volumes and surface areas in $\mathbb{R}^n$ measured
with this norm) and seems to be proper for the multi-dimensional
case; we will avoid thus to define the total price by different
measures, like for example the average value of all assets, or
the sum of the $n$ prices, which may fit better to one-dimensional
approaches.
\end{remark}

We shall consider the following stochastic outer Stefan problem
for the Heat equation with noise
\begin{equation}\label{stef}\left\{
\begin{aligned}
\partial_t w=&\alpha\Delta w
+\sigma({\rm
dist}(x,\Gamma))\dot{W}(x,t),\;\;x\in\mathbb{R}^n-\mathcal{D}(t)\;\;\mbox{(`liquid' phase)},\;\;t>0,\\
w=&w_0=0,\;\;x\in\mathcal{D}(t)\;\;\;\;\;\;\;\;\;\;\;\;\;\;\;\;\;\;\;\;\;\;\;\;\;\mbox{(`solid' phase)},\\
w=&-k+w_0\;\;{\rm on}\;\;\Gamma(t)\;\;\;\;\;\;\;\;\;\;\;\;\;\;\;\;\;\mbox{(Gibbs Thomson condition)},\\
V=&-\nabla w \cdot \eta\;\;{\rm
on}\;\;\Gamma(t)\;\;\;\;\;\;\;\;\;\;\;\;\;\;\;\;\;\mbox{(change of
liquidity}\\
&\mbox{\;\;\;\;\;\;\;\;\;\;\;\;\;\;\;\;\;\;\;\;\;\;\;\;\;\;\;\;\;\;\;\;\;\;\;\;\;\;\;\;\;\;\;\;\;\;\;\;driven by the strength of trade,}\\
&\mbox{\;\;\;\;\;\;\;\;\;\;\;\;\;\;\;\;\;\;\;\;\;\;\;\;\;\;\;\;\;\;\;\;\;\;\;\;\;\;\;\;\;\;\;\;\;\;\;\;also called Stefan condition)},\\
\Gamma(0)&=\Gamma_0,
\end{aligned}
\right.
\end{equation}
where $k$ is the mean curvature of $\Gamma$, $V$ the velocity of
$\Gamma$, $\dot{W}(x,t)$ a space-time noise, and $\sigma$ is a
noise diffusion. Moreover, $\alpha>0$ is the positive constant
coefficient of the Laplacian operator modeling the diffusion of
the trading that stabilizes market's variations, cf. also in
\cite{zi}. We also impose a condition at infinity (far-field
value) of the form
$\displaystyle{\lim_{r\rightarrow\infty}}w(r,t)=w_{\infty}(t)$.

The coefficient $\alpha$ reflects the liquidity of the market, and
will be referred as \textit{liquidity coefficient}. An increasing
value for $\alpha$ implies that more intense active trading
occurs, and thus, it is expected that the solid phase (for
example the spreads domain in a case of interest) will become
smaller and will reach at an equilibrium earlier in time. For
simplicity, we assume that $\alpha$ remains constant for any
$t\in[0,T]$ which is a reasonable assumption, when evolution is
observed in short time intervals, as for example during a day.

The noise diffusion $\sigma$ is a \textit{volatility} that
depends on the distance of the prices vector $x$ from the
liquidity boundary $\Gamma$. In dimension one, in
\cite{phd2,main1}, the authors proposed a volatility of the form
$\sigma=\sigma(|x-S^*(t)|)$ for $S^*(t)\in\mathbb{R}$ the mid
price of one share from the limit order book, and $x\in\mathbb{R}$
its price in a logarithmic scale. This represents the dependence
of the noise strength on the distance of the current price $x$
from an average price $S^*$ (mid price there), which is equal to
$|x-S^*|$, when the spread is zero. The analogous argument for a
model permitting non zero spreads, in dimensions $n$ where the
distance is measured by the euclidean norm in $\mathbb{R}^n$,
leads to a volatility definition of the form
$$\sigma=\sigma({\rm
dist}(x,\Gamma(t))).$$ The volatility of an asset is a measure of
the dispersion of the prices of the asset as it evolves in time.
If the price remains stable the volatility is low and the risk
for holding (not trading) the asset is low. In the case of a
portfolio the volatility can be generalized to be a measure of
risk of the investment. Let us assume a price in the liquid
phase; as the distance of the price vector from the boundary of
the liquid phase increases the risk for not making a transaction
increases as well, and the volatility achieves a higher level.

In the proposed problem, by definition, in the solid phase zero
trading occurs, and therefore, the density of trading therein is
zero for any time $t$, i.e., $w_0(t)=0$.
\begin{remark}
Even if not analyzed in this paper, we point out that in various
one-dimensional Stefan problems for limit orders (applied under
the simplification assumption of a zero spread though) a usual
choice for the density $w_0$ is again zero. However, in these
approaches (due to the vanishing spread), $w_0=0$ only models the
immediate execution of the sell or buy orders from the limit order
book of one share (when decided), and seems feasible enough, due
to the direct computerized interaction of the network of various
trade markets, see in \cite{phd2}. A positive constant $w_0>0$
could be also considered in dimensions one for these cases, when
a certain delay of transactions is inserted in the model. Note
that a simple change of variables of the form $w\rightarrow
w-w_0$, as the spde
 and the Stefan condition for the velocity are linear, leads to a
zero delay model.
\end{remark}

The Stefan b.c. describes the velocity of the interface
$\Gamma(t)$, which of course is given by the jump of the gradient
of the density $w$ along $\Gamma$. In our problem, in the solid
phase $w=w_0={\rm const}=0$, so the jump involves only the
gradient of the density in the liquid phase, since the other term
in the difference is vanishing, cf. also \cite{nietthes}. A
Stefan b.c. of this kind, where the velocity is given by the jump
at a mid price (the moving boundary then consists of a moving
point on a line), has been already proposed in \cite{main1,zi},
for a system of $2$ equations for sell and buy orders
respectively in dimension one, and describes the change of the mid
price driven by the strength of the ask price. In our case it is
the liquidity area that changes and this change is driven by the
strength of trading since the evolution of the (total) density $w$
(volume of transactions in sell and buy orders) for all the $n$
assets is given by one equation but posed in dimensions $n$.

A detailed motivation for the Gibbs-Thomson b.c. condition and
its financial interpretation will be presented at a separate
section, in the sequel.

Two-phases elliptic Stefan problems with analogous b.c. appear as
the sharp interface limit of the Cahn-Hilliard equation,
\cite{abc}, or the stochastic limit of Cahn-Hilliard equation
with noise, \cite{abk}. Moreover, considering Allen-Cahn or the
stochastic Allen-Cahn equation, the law of motion on the sharp
interface limit is described by a velocity given by the mean
curvature or stochastic mean curvature respectively (and not by
the jump); see for example the classical results of Evans, Soner,
Souganidis in \cite{ess} for the deterministic equation, and this
of Funaki for the stochastic case with mild noise, \cite{Fun99}.

Let us describe the mathematical statement of \eqref{stef} in
terms of a moving boundary problem. It is a one-phase outer
parabolic Stefan problem, since the parabolic type spde
(Stochastic Heat equation) is posed only on one phase, the liquid
phase, placed outside the solid one. At the initial time $t:=0$,
the solid phase $\mathcal{D}_0:=\mathcal{D}(0)$ is considered
already formed as a bounded domain in $\mathbb{R}^n$ and thus,
its boundary $\Gamma_0:=\Gamma(0)$ is given. Obviously, due to
boundedness, $\Gamma_0$ is a closed hyper-surface of
$\mathbb{R}^{n-1}$, embedded in $\mathbb{R}^n$. As it is usual to
Stefan problems from phase separation, $\Gamma_0$, in a more
general setting, is a union of such surfaces, \cite{niet2,abc}.
So, at the initial time, \eqref{stef} is fully determined by one
Stochastic Heat equation posed on the unbounded domain
$\mathbb{R}^n-\mathcal{D}_0$ with non-homogeneous Dirichlet b.c.
on the boundary $\Gamma_0$ involving the mean curvature of
$\Gamma_0$. The solution $w$ of the above, through $\nabla w$,
defines then the velocity $V$ of the moving boundary $\Gamma(t)$,
and therefore its evolution and shape at a next time, and this
determines the new unbounded domain with boundary $\Gamma$ where
the spde is posed, and so on.

In \cite{nietthes}, the initial boundary $\Gamma_0$ has been
assumed to consist of a union of spherical surfaces; therein, in
the deterministic setting and in dimensions $n=3$, the same
problem \eqref{stef}, for $\alpha>0$ has been considered but for
$\sigma=0$, and for a different application from material
science, the Ostwald Ripening of alloys. In particular,
Niethammer in \cite{nietthes} analyzed a mean field approximation
model where the solid phase is a union of spherical domains with
fixed centers and evolving radii, and derived the dynamics of
radii. Moreover, she proved  well posedeness for the static
(elliptic) problem with undercooling, \cite{niet2}. Antonopoulou,
Karali and Yip in \cite{JDE12} proved well posedeness for the
full parabolic problem with undercooling and obtained the
modified dynamics of radii; cf. also the work of X. Chen and
Reitich for the two-phases deterministic Stefan problem,
\cite{XFC1}, and in \cite{af, afk1, afk2} for a quasi-static
version.

\begin{remark}
The existing rigorous mathematical literature on well posedeness
and dynamics of multi-dimensional two-phases Stefan type problems
(cf. for example \cite{abc,XFC1} and the references therein),
concerns so far the deterministic problem where the same exactly
pde is posed on the two different phases, i.e. on
$\mathbb{R}^n-\Gamma$. However, if the trading is to be
classified in sell and buy orders corresponding to a liquid and a
complementary solid phase respectively, the model would demand a
system of $2$ equations of the form appeared in \eqref{stef} with
different parameters, and it has been very recently analyzed and
only in one dimension, \cite{zi,main1}. In higher dimensions,
$n\geq 2$, there exist many open questions on existence,
regularity and dynamics for Stefan problems posed as a system of
two equations even in the absence of noise.
\end{remark}
\subsection{The spherical boundaries stochastic Stefan model for $n=3$}
 In the general stochastic Stefan problem \eqref{stef} we set $n=3$. So, we consider that the initial solid phase
 is in $\mathbb{R}^3$. Moreover we assume that the initial solid phase is the union of $I$
spherical domains, and, as in \cite{nietthes}, that during
evolution the centers remain constant.

So, we define $$\mathcal{D}(t):=\displaystyle{\cup_{i\in
I}}B_{R_i}(t),$$ for $B_{R_i}(t)$ a ball of radius $R_i(t)$ and
fixed center $x_{c}^i\in\mathbb{R}^3$ for $i \in I$. The boundary
$\Gamma(t)$ at time $t$ is the union of the $I$ spherical
boundaries
$$\Gamma_i(t):=\partial B_{R_i}(t),$$
and so, given by
$$\Gamma(t):=\displaystyle{\cup_{i\in
I}}\Gamma_i(t).$$

The problem \eqref{stef} is transformed into $I$ problems, given
for each $i\in I$ by
\begin{equation}\label{stef2}
\begin{split}
\alpha^{-1}\partial_t v(x,t)=&\Delta v +\alpha^{-1}\sigma({\rm
dist}(x,\Gamma))\dot{W}(x,t),\;\;x\in\mathbb{R}^n-\mathcal{D}(t)\;\;{\rm('liquid'\;
phase)},\;\;t>0,\\
v=&0,\;\;x\in\mathcal{D}(t)\;\;\;\;{\rm('solid'\;phase)},\\
v=&\frac{1}{R_i(t)}\;\;{\rm on}\;\;\Gamma_i(t),\;\;\;\;\;\;\;\;\;\;\;\;\;\;\;\;\;\mbox{(Gibbs Thomson condition),}\\
\dot{R}_i(t)&=\frac{1}{4\pi R_i^2(t)}\int_{\Gamma_i(t)}\nabla v\cdot \eta,\;\;\;\;\;\mbox{(Stefan condition),}\\
\Gamma(0)&=\Gamma_{0},
\end{split}
\end{equation}
where, as in \cite{nietthes}, we applied the transformation
$$v:=-w,$$ replaced the curvature of the sphere by the inverse of its radius, and integrated the last b.c. along the spherical
boundary; for the general transformation see at the first lines
of pg. 125 of \cite{nietthes}, in particular, we took $H=K$, and
all appearing constants equal to $1$ except of $T_0$ taken as
$T_0:=0=w_0$ and of $C$ taken as $C:=\alpha^{-1}$. Also see the
statement of the transformed problem at pg. 127 of
\cite{nietthes}, for zero volatility $\sigma$. Note that
$\dot{R}_i$ denotes the time derivative of the $i$ radius. The
far-field value takes the form
$\displaystyle{\lim_{r\rightarrow\infty}}v(r,t)=v_{\infty}(t)$ and
$v_\infty(t)$ consists one of the unknowns of the problem.

When volatility is not a vanishing quantity, then the integral of
the b.c. along $\Gamma_i(t)$ is formally taken, assuming that the
sphere remains a sphere during evolution, but with stochasticly
fluctuating radius.

\begin{remark}
A ball in $\mathbb{R}^n$, is the logarithmic image of a bounded
simply connected domain in the initial coordinates in
$\mathbb{R}^{n+}$, where the initial real values of the portfolio
are set.
\end{remark}
\subsection{Financial interpretation of the Gibbs Thomson
condition} Ostwald in \cite{os}, first observed that during the
late stages of phase separation also called as coarsening, the
evolution favors the minimization of surface energy of the inner
interfaces separating the phases. Considering the case of
liquid/solid phase transitions, the previous is translated to the
reduction of the surface area of the solid phase, where the
diffusional mass (measured by the integral of our density
solution in the liquid phase) is transferred from regions of high
interfacial curvature to regions of low interfacial curvature,
\cite{nietthes}. The Gibbs Thomson condition of problem
\eqref{stef} involving the curvature $k$ is an effective
approximation of the above growth law and is extensively used to
the literature of multi-dimensional Stefan problems, where the
geometric characteristic of the curvature of curves, $n=2$, or
surfaces $n\geq 3$, (which form the phase separation sharp
interfaces), has a meaning.

When the solid phase at time $t$ consists of $2$ well separated
spherical domains of radii $R_1(t)>R_2(t)$ and thus of curvatures
$\frac{1}{R_1(t)}<\frac{1}{R_2(t)}$, in later times as separation
evolves, the growth of the larger sphere is expected (here this
of radius $R_1$) at the expense of the smaller. In fact this is
rigorously proved for $n=3$ and zero volatility in
\cite{nietthes,JDE12} for the Stefan problem of type
\eqref{stef2}, and for a more general case where kinetic
undercooling acts on the Gibbs Thomson condition.

Moreover in the case of solid phase of a more complex geometry,
the aforementioned optimization constraint set by the growth law
leads during evolution to minimizing area moving boundaries close
to spheres, cf. \cite{afk1,afk2}.

In the financial setting, let us consider that at our initial
time, the solid phase consists of $2$ well separated balls with
centers two marginal estimations, or even real but different
observations for the vectorial price $x$ of our portfolio, as for
example when two markets are participating by trading all the
$n=3$ assets during the same period of one financial day; $n=3$
currencies in European Union is a basic case.

In accordance to \eqref{stef2} notation, we define the fixed
centers by
$$x_{c}^1,\;\;\;\; x_{c}^2,$$ and the initial radii by
$$R_1(0)<R_2(0),$$ small enough, which as we shall analyze in detail in a following section
represent the half of the minimum of the $n$ spreads respectively
at the given initial time, and obtain the well separated balls
condition
\begin{equation*}\label{ni}
\|x_{c}^1-x_{c}^2\|>>R_1(0)+R_2(0).
\end{equation*}

\begin{remark}
The scaling of the Stefan problem \eqref{stef2} is of significant
importance, this being related to the mean field assumption of an
initial solid phase consisting of $I$ well separated spherical
domains with relatively small radii, that do not touch during
evolution; a result is a comparatively very large magnitude for
the liquidity coefficient $\alpha$.

For example, see also in \cite{nietthes} for an analogous
condition after rescaling, in dimensions $n=3$ and for
$\sigma=0$, the condition \eqref{volfr}, which we will analyze
more extensively in a following section, must hold, i.e.
\begin{equation*}\label{volfr}
I\displaystyle{\max_{i=1,\cdots,I}}R_i(0)<<\mathcal{O}(\alpha^{4/9}).
\end{equation*}
\end{remark}

As time passes, for the problem \eqref{stef2}, and when
$\sigma:=0$, the theory predicts that the smaller ball of radius
$R_1(0)$ will begin to shrink while the other will grow. This
means that the smaller (minimum) spread will be reduced in the
first market while the larger (minimum) spread will increase in
the second market. The above is indeed expected since small
spreads are observed to highly traded assets and tend thus to
reduce. On the other hand a comparatively large spread is an
index of low trading and of higher risk for the investor.

\begin{remark}
Considering the problem \eqref{stef2} ($n=3$), for $\sigma=0$, for
$I$ initial balls, there exist the so-called vanishing times
$t_{{\rm v}i}$ for the spherical domains constituting the solid
phase, \cite{nietthes}. A case of interest is the equilibrium
where only one of the $I$ balls survives and stands as the final
solid phase, while its diameter approximates a maximum spread
value in the time interval $[0,T]$. Here, $T$ is equal to the
last vanishing ball time. However, there exist equlibria of more
than one balls of equal radii.

In the case of two balls for the initial solid phase ($I:=2$),
with radii ordered as follows
$$R_1(0)<R_2(0),$$
since the smaller will eventually vanish, let us say at
$t:=t_{\rm v}$, the maximum spread in $[0,T]$ coincides to
$$2R_2(t_{\rm v}).$$ The above, gives a useful
prediction for the future optimal investment of the portfolio.

The evolution of the $I$ radii, and thus, their values and
vanishing times are well estimated from the approximating dynamics
for the radii which are given in a following section by the ODEs
$\eqref{1.6}$, $\eqref{1.5*}$.
\end{remark}

\begin{remark}
The quasi-static version of the parabolic problem \eqref{stef2},
assumes a diffusion coefficient $\alpha\rightarrow\infty$. For
this model and for $\sigma:=0$, volume conservation holds for the
solid phase, see Lemma 2 at pg. 135 of \cite{nietthes}. Thus, the
diameter of the last surviving ball is given by
\begin{equation}\label{i1}
2\Big{(}\displaystyle{\sum_{i\in I}}R_i(0)^3\Big{)}^{1/3},
\end{equation}
which is the largest observed spread and depends strongly on the
initial definition of the solid phase, in particular on the
initial radii, i.e. the initial spreads.
\end{remark}
\subsection{A liquidation strategy} Efficient strategies for
portfolios management are based on the quantification of the
uncertainty of prices and of market liquidity. The portfolio
optimal performance is restrained by the control of the variance
of the expected returns under minimum risk investment policies
where a certain utility function is maximized, \cite{mark,merton}.

In \cite{mark}, an investor wants to allocate his initial amount
among a given number of assets where the expectations of returns
are taken as known. The criterion for determining the set of
optimal portfolios, that is, the optimal weights (proportions of
total wealth) assigned to the assets, is to minimize the variance
of the expected returns. However, the limitation of this approach
is that the expected returns (which are estimated by financial
data) are assumed to be constant in time; this is also called the
static optimization problem.

Merton, in \cite{merton}, performs a continuous-time analysis for
the problem of optimal portfolio selection where the rates of
return of the individual assets are generated by a Wiener
Brownian motion process. The optimal proportions of total wealth
(or weights) that are invested in each asset for any given time
is derived by maximizing its expected utility as a function of
wealth.

As we shall describe, the solution of the Stefan problem
\eqref{stef} with properly defined parameters can contribute as a
recommendation tool to an investor who already holds a portfolio
of assets and wishes to liquidate to cash some fraction of each
one of them during a time interval $(0,T)$.

We consider a financial market of $n$ risky assets with prices per
share $$p_i(t)\in\mathbb{R}^+,\;\;i=1,\cdots,n,$$ at time $t$. An
investor holds a portfolio of these assets with allocations
$$s(t)=(s_1(t), \cdots, s_n(t));$$ here, $s_i(t)$
is the number of shares of the $i$ asset at $t$. So, the value of
the portfolio at time $t$ is given by
\begin{equation}\label{m2}
\mathcal{V}(t)=\sum_{j=1}^{n}{s_j(t)}p_j(t).
\end{equation}
When the portfolio is liquidated, $p_i(t)$ can be specified on
real time by the limit order book of asset $i$, or be predicted
in advance. In general, $s_i$ varies in time.

Let $f_i(t)$ denote the fraction of the initial amount of asset
$i$ that the investor wants to sell at time $t$ ($f_i \in [0,1]$).
The allocation of asset $i$ at time $t$ can be modeled by
\begin{equation}\label{md3}
s_i(t):=s_i(0)-f_i(t)s_i(0)=(1-f_i(t))s_i(0),
\end{equation}
 for $s_i(0)$ some initial given
allocation of asset $i$ at the initial time $t=0$. Here, we
consider that, for any $i$, $f_i$ are defined to satisfy
$f_i(0)=0$. This implies that at the initial time the investor
will never choose to sell any share of his portfolio. Therefore,
by replacing $s_i$, the investor's portfolio allocation is given
by the vector
\begin{equation}\label{m4}
s(t)=\Big{(}s_1(0)-f_1(t)s_1(0), \cdots,
s_n(0)-f_n(t)s_n(0)\Big{)}.
\end{equation}

We consider as time $t \in [0,T]$ the first instant that the
investor sells parts of his portfolio, and thus, no transaction
has been performed in the interval $[0, t)$. Time $t$ is a part of
investor's strategy that will be derived based on information
offered by the evolution of the whole market as shown in the next
section. As soon as a transaction is performed at time $t$, the
model is initiated and the time is set to $0$ again. So, a new
period $[0,T]$ starts for the investor for future transactions.

 This trading activity will affect the portfolio performance
in terms of returns and risk. Liquidation strategies are
developed and applied in order to ensure that the remaining
portfolio will have a high rate of return.

We fix a time $t$, and define
\begin{equation}\label{m5}
C(t):=\sum_{j=1}^{n}{f_j(t)s_j(0)}p_j(t),
\end{equation}
as the amount of consumption resulted by the liquidation of the
portfolio.

At the asset $i$, for any $i=1, \cdots, n$, we assign at time $t$
the non-negative weight $z_i=z_i(t),$ defined as
\begin{equation}\label{m1}
z_i(t)=\frac{s_i(t)}{\displaystyle{\sum_{j=1}^{n}}{s_j(t)}},
\end{equation}
 so that
$$\displaystyle{\sum_{i=1}^{n}}{z_i(t)}=1,\mbox{ and }z_i(t)\in[0,1].$$
Let $\mathcal{R}(t)$ be the rate of return of the remaining
portfolio at time $t$; $\mathcal{R}(t)$ is defined as
\begin{equation*}\label{m6}
\mathcal{R}(t):=\sum_{i=1}^{n}{z_i(t)\frac{p_i(t)}{p_i(0)}}=\sum_{i=1}^{n}
{\frac{(1-f_i(t))s_i(0)}{\displaystyle{\sum_{j=1}^{n}}(1-f_j(t))s_j(0)}\;\frac{p_i(t)}{p_i(0)}},
\end{equation*}
where we replaced the weights $z_i$ by using \eqref{m1} and
\eqref{md3}.

We define as utility $U$ of the investor a measure that captures
the satisfaction he obtains when involved in trading activities
concerning his portfolio. More precisely,
\begin{equation}\label{uti}
U=U(\mathcal{V}(t), C(t)),
\end{equation}
is assumed to be a strictly concave function of the value
$\mathcal{V}$ of his portfolio, and of the consumption level $C$
that is liquidated at time $t$.
\begin{definition}\label{def2}
We define the liquidation strategy at time $t$ to be the vector
of fractions $$f(t)=(f_1(t), \cdots, f_n(t)),$$ for $f$ the
solution of the following maximization problem with constraint:
\begin{equation}\label{max}
\begin{split}
&\max_{f}\,{U(\mathcal{V}(t),C(t))},\\
{\rm s.t.}\;\;\;\;\;\;\;\;
\sum_{j=1}^{n}&{s_j(0)}-\sum_{j=1}^{n}{s_j(t)} \le w^*,
\end{split}
\end{equation}
for $w^*$ the level of the available total volume of all the
shares of the assets in the portfolio. Here, remind that the
fractions $f_i$ appear in the formulae of $\mathcal{V}$ and $C$,
cf. \eqref{m2}, \eqref{m4}, \eqref{m5}, while $U$ is the utility
function given at \eqref{uti}.
\end{definition}

In order to solve the optimization problem \eqref{max} (which is
not in the aims of the current work), one has first to estimate
the prices vector $p(t)=(p_1(t), \cdots, p_n(t))$ of the shares
which appear in the definition of $\mathcal{V}(t)$ and $C(t)$,
and $w^*$. The solution of the Stefan problem \eqref{stef2}, in
particular the moving boundary, provides a prediction in a
logarithmic scale (for example the spreads) for the price vector
$p(t)$ at a given time $t$. This information is crucial for
maximizing the utility function.


\bigskip

\bigskip
Evidently, the liquidity parameter $\alpha$ and the choice of
initial solid phase, even if defined mathematically, when replaced
should be related to a specific financial application since they
concern the market's characteristics.

Our aim in the following section is to address a main financial
application of the Stefan problem as stated in \eqref{stef2},
which is posed in a logarithmic scale for the spatial coordinates
$x$, where the trading (i.e., the diffusion of the density) is
observed to one or more financial markets.

The spatial coordinates defining $x=x(t)=(x_1,\cdots,x_n)$ will
correspond to the prices of trading (sell/buy) at time $t$ of $n$
different shares, while the solid phase diameter will approximate
the minimum of the $n$ spreads for orders from the limit order
book.

We shall properly define the prices $x_i$, $i=1,\cdots,n$ and the
initial data of \eqref{stef} in the version \eqref{stef2}:
$\alpha$, $\mathcal{D}(0)$ (the initial solid phase which is
related to zero trading areas), and $\Gamma(0)$ (the initial
solid phase boundary). In addition, we shall present carefully the
financial interpretation of all these parameters.

\section{Sell/Buy orders and spreads from the limit order book}
\subsection{Preliminaries} In portfolio selection, the
investor uses financial data such as expected prices, rates of
return, market liquidity and many other; these parameters are
often estimated by historical data provided by the limit order
books of the assets of interest.

The evolution of market sell or buy limit orders for a particular
asset placed by investors in a financial market is described in
the limit order book, \cite{lob}. At any time $t \in[0, T]$, the
limit order book contains a list of sell and buy limit orders for
an asset, and it is continuously updated in $[0, T]$. The
information contained in an order book is significant for
discovering the price of an asset, and affects substantially the
investors' decisions on choosing optimal trading strategies.

A trading strategy (or limit order) is characterized by three
components: the time to place the order, the quantity of shares
that is for trade and the limit price per share.

In particular, a sell (buy) order is placed in the $i$ limit order
book (which is denoted by $LOB_i$), when an investor wants to sell
(buy) a specified number of shares of asset $i$ at or over
(below) a specified price; this price is called \textit{limit
price}.

Let $A_i(t)$ be the \textit{ask price} which is the lowest sell
order (i.e., the minimum price at which the investor is willing to
receive), and let $B_i(t)$ be the \textit{bid price} which is the
highest buy order (i.e., the maximum price at which the investor
is willing to pay), both contained in the order book. The ask
price is always higher than the bid price. Thus, a sell order that
arrives at time $t$ is executed, more specifically the asset is
sold, if the associated price being set by the investor is lower
than the current bid price at time $t$. Otherwise, the sell order
 is sorted in the list of the order book.

The average of the the ask and bid prices of the asset $i$ at time
$t$
\begin{equation}\label{m9}
\bar{p}_i(t):=\frac{A_i(t)+B_i(t)}{2},
\end{equation}
 is called
 \textit{mid price}, while the difference
\begin{equation}\label{m10}
spr_i(t):=A_i(t)-B_i(t),
\end{equation}
 between the ask and bid prices
at time $t$ defines the \textit{spread} for the order book of the
asset $i$.

\begin{remark}
The spread reflects the liquidity of the asset. Liquidity is a
measure that describes how quickly the asset is traded. For
example, a high liquidity asset is cash or currency, while a low
liquidity asset is art or real estate. An overview of indicators
that can be used to measure liquidity can be found in \cite{liq}.
The dependence of the spread of an asset with its liquidity
indicates an inverse relation: a wide spread implies a low liquidity asset.
\end{remark}

For a portfolio of $n$ different shares with prices
$p_i\in\mathbb{R}^+$, $i=1,\cdots,n$, let
$$x=(x_1,\cdots,x_n):=(\ln(p_1),\cdots,\ln(p_n))\in\mathbb{R}^n.$$

Let $w=w(x,t)$ (in accordance to the notation used in
\eqref{stef}) be the \textit{fluctuating density, cf.
\cite{main1}, or volume, cf. \cite{phd2}} of the \textit{limit
sell and buy orders} of all $n$ assets placed at price $x$ (i.e.,
corresponding to $p:=(p_1,\cdots,p_n)$ before the logarithmic
scaling).

Referring to the initial coordinates $p_i$, the zero trading
domain, contains all the possible prices in an $n$-dimensional
open rectangular domain $\mathcal{S}$ induced by the ask and bid
prices of each of the $n$ shares. More specifically, for all
prices $p_i$, $i=1,\cdots,n$, being lower than the respective ask
prices, and higher than the respective bid prices contained in
the limit order books at time $t$, no trading is possible. The
edges of the rectangle have lengths equal to the spreads
$spr_1(t),\cdots,spr_n(t)$, since for each $p(t) \in \mathcal{S}$
the coordinate $p_i(t)$ is within the interval $[B_i(t), A_i(t)]$
of the respective order book. Obviously, the mid price
$\bar{p}_i(t)$ of each asset is the midpoint of  $[B_i(t),
A_i(t)]$, and the `center' of $\mathcal{S}$ at the same time $t$
is given by the coordinates of the mid prices vector
$\bar{p}(t):=(\bar{p}_1(t),\cdots,\bar{p}_n(t)).$

Note that when more (interacting) markets are considered at the
same time $t$, the zero trading area consists of more than one
domains, defined by the corresponding spreads and mid prices for
the same shares taken from the limit order books of the different
markets.

At any price vector $p$ outside $\mathcal{S}$ there is a
possibility of trading (either sell or buy), and the volume of
the portfolio at this price indicates the total number of shares
that may be sold or bought. So, we may have one of the following
trading activities depending on the position of this $p$:
\begin{enumerate}
\item Sell opportunities for all or some assets.
\item Buy opportunities for all or some assets.
\item Sell opportunities for some assets and buy opportunities for other.
\end {enumerate}
Of course since the evaluation of spreads is observed in discrete
times, the aforementioned definition of zero trading areas is an
idealized one. In practice the boundary of $\mathcal{S}$, as
defined, will include all price vectors that are the most
favorable for obtaining the available shares of the assets, if the
respective sell/buy orders are executed. Roughly speaking, in
$\overline{\mathcal{S}}$ (i.e., the boundary included) the volume
of trading is minimized; we mention that on the boundary, the
prices optimize trading: i.e. for a sell order the
price on the boundary results in higher profits (though the probability of trade is decreased).

It is expected that the higher the distance of prices vector $p$
from the spreads area is, the higher the trading is. Note that
the rate of change of the trading volume of an asset may vary
significantly in the time interval $[0,T]$; for example, as
frequently observed, there is a decrease of trading during lunch
time. In addition, this rate may be influenced by the impact of
trading activities involving the same asset in other financial
markets, or when new information arrives about the asset.
\begin{remark}
In our approach, we shall consider one differential equation for
both sell and buy orders assuming that the demand and the supply
of the assets in the portfolio evolve according to a single
parameter $\alpha$ that is related to the total liquidity of the
markets.
\end{remark}

\subsection{Solid phase of spherical domains with varying radii
and constant centers} As already mentioned, $p(t)$ is the price
vector of the portfolio at time $t$. Remind that the portfolio
consists of $n$ different shares.

We apply the change of variables $$x=(\ln(p_i),\cdots,\ln(p_n)),$$
for the new coordinate system in space. At time $t$ the solution
of the Stefan problem defined in \eqref{stef} provides the total
density $w(x,t)$ and the solid phase $\mathcal{D}(t)$ at which
$w=0$.

Motivated by the model described by Niethammer in \cite{nietthes}
for zero volatility, where a union of spherical domains
constitutes the initial solid phase, while these domains remain
spherical, we approximate the initial solid phase
 at time $t=0$ with a spherical one centered at the
 rescaled coordinates of the initial mid price vector
 \begin{equation}\label{exmp}
 \bar{p}(0)=\Big{(}\frac{A_1(0)+B_1(0)}{2},\cdots,\frac{A_n(0)+B_n(0)}{2}\Big{)}=:(\bar{p}_1(0),\cdots,\bar{p}_n(0)),
\end{equation}
  i.e. for $I=1$ the center is given by
 \begin{equation}\label{exc}
x_c=(x_{c1},\cdots,x_{cn}):=(\ln(\bar{p}_1(0)),\cdots,\ln(\bar{p}_n(0))).
\end{equation}
The analogous approach can be applied for $I\geq 1$, by using the
data of the limit order books of each $i=1,\cdots,I$ market.

The radius of this initial spherical solid phase $\mathcal{D}(0)$
is defined by
\begin{equation}\label{rad}
R(0):= \min_{i=1,\cdots,n}{\frac{lspr_i}{2}},
\end{equation}
for
\begin{equation}\label{lsp}
lspr_i:=\ln(A_i(0))-\ln(B_i(0)).
\end{equation}

 So, we exclude the larger spread values; this is a reasonable
strategy when in our portfolio assets of analogous spreads are
considered as for example currencies. Moreover, for small spreads
of order $\mathcal{O}(10^{-1})$ - $\mathcal{O}(10^{-4})$, which
is the usual case for assets of high liquidity, even in the
logarithmic scale, the minimum spread enclosed area is a good
approximation of the solid phase; see for example the following
data for the British Pound versus US Dollar currency taken in
March 2019, \cite{ad}, 19 March 2019 British Pound v US Dollar
Data Latest GBP/USD: Exchange Rate: 1.3275, \textit{Bid: 1.3275,
Ask:  1.3276}, Market Status: Live, Percent Change: +0.0939,
Today's Open (00:01 GMT): 1.3262, Today's High: 1.3309, Today's
Low:  1.3241, Previous day's Close (23:59 GMT): 1.3263, Current
Week High: 1.3309, Current Week Low: 1.31841, Current Month
High:  1.33786.

The center of the spherical domain represents the mid vectorial
price of the $n$ assets and the radius represents the range of
the minimum spread of all $n$ assets around the mid price.

We also assume that the center of the spherical domain remains
constant in time, that is the mid price of each order book does
not change in a small time horizon (e.g. within a day).

We can extend our model of one financial market to the scenario
where the investor is interested in taking part in more than one
markets for the same portfolio. This is translated to considering
more than one domains of different radii $R_i(t)$, one for each
financial market, and as solid phase the union of them.
 When significantly different mid prices per market and relatively
small spreads occur, the initial spherical domains can be assumed
well separated and placed far enough one from the other so that
during evolution they do not touch. This is in accordance to the
necessary assumption for the deterministic Stefan problem model
of \cite{nietthes} where the theory predicts the increase of the
larger spherical domain at the expense of the smaller, at least
when the volatility is zero. In such a case, the coefficient
$\alpha$ of the SPDE of \eqref{stef2} will be related to the
total liquidity of the different markets.

In the above, our aim was to approximate the initial solid phase
by a spherical domain of diameter the minimum spread value that
theoretically can be taken at the specific initial time $t_0$
from the limit order books. However, in realistic cases, the
provided data are discrete. We define instead an average value of
historical data very close to the initial time $t$. This is
implemented by using the data of the following Definition
\ref{mdef}.
\begin{definition}\label{mdef}
Let $t_1, \cdots,t_m$ be $m$ time instants prior to time interval
$[t_0,t_0+T]$ (where $t_0$ is the initial time). We will use the
following data for each order book $LOB_i,\;i =1,\cdots,n$ at
each time $t_j,\; j=1,\cdots,t_m$: the ask price $A_i(t_j)$, the
bid price $B_i(t_j)$, the spread $spr_i(t_j)= A_i(t_j)- B_i(t_j)$
and the total volume $w_i$ of sell orders for each asset $i$ that
have been executed additively during the $m$ time instants.
\end{definition}

More precisely, in order to estimate the mid price and the spreads
of the $n$ assets at $t_0$ (which will be set to $0$), we define
the average version of \eqref{exmp}, where an average initial mid
price vector is used, given by
\begin{equation}\label{exmpa}
\bar{p}_a:=\Big{(}\displaystyle{\sum_{j=1}^{m}}\frac{A_1(t_j)+B_1(t_j)}{2m},
\cdots,\displaystyle{\sum_{j=1}^{m}}\frac{A_n(t_j)+B_n(t_j)}{2m}\Big{)}=:(\bar{p}_{a1},\cdots,\bar{p}_{an}).
\end{equation}
This will define the center of the solid phase at $t_0:=0$ by the
logarithm of its coordinates, i.e., the definition \eqref{exc} is
replaced by
 \begin{equation}\label{exca}
x_c:=(\ln(\bar{p}_{a1}),\cdots,\ln(\bar{p}_{an})).
\end{equation}
In this case the initial radius at $t_0:=0$ is given by the
following averaged version of \eqref{rad}
\begin{equation}\label{rada}
R(0):= \min_{i=1,\cdots,n}{\frac{lspra_i}{2}},
\end{equation}
for
\begin{equation}\label{lspa}
\begin{split}
lspra_i:=&\ln\Big{(}\displaystyle{\sum_{j=1}^{m}}A_i(t_j)/m\Big{)}-\ln\Big{(}\displaystyle{\sum_{j=1}^{m}}B_i(t_j)/m\Big{)}\\
=&\ln\Big{(}\displaystyle{\sum_{j=1}^{m}}A_i(t_j)\Big{)}-\ln\Big{(}\displaystyle{\sum_{j=1}^{m}}B_i(t_j)\Big{)}.
\end{split}
\end{equation}

We also assume that the fixed cost of liquidation per share is
equal to a fixed rate of general transaction costs ($H=K$ for the
physical problem, which appears in the Stefan condition for the
velocity).

\subsection{The liquidity coefficient $\alpha$}
Estimations of various characteristics of the market, such as the
level of liquidity, the assets evaluation, and the volume of
trading activity per temporal period (for example during a
financial year) consist the context of the limit order book as we
discussed in the previous section; see for example the relevant
survey presented in \cite{surv}. The limit order represents the
trade of a specified amount of an asset at a predetermined price.

The Laplacian coefficient $\alpha>0$ of the SPDE of Stefan problem
\eqref{stef} (appearing also in \eqref{stef2}), measures the
diffusion strength of sell and buy orders during trading; remind
that for the financial application considered in this section,
the total trading is observed, and so, in liquidation sell and as
well buy orders participate.

A large value $\alpha>>0$ models a high-volume market with intense
trading activity.

We shall assume that $\alpha$ is constant in a short period of
one day, this implying that the tendency of the demand of the
market concerning the assets of interest will not change its
pattern. Since $\alpha$ reflects the liquidity of the market, we
expect that it will increase as the number of the shares of the
assets that are traded (sold and/or bought) in the recent past is
increased. Also a low spread shows a tendency of the market to
face a high liquidity.

Taking these remarks into account we define $\alpha$ to be a
weighted average of the liquidity measures of each asset
separately which will be denoted by the symbol $\alpha_i$.

We shall use the following historical data from the order book of
each of $n$ shares of Definition \ref{mdef}.

Let
\begin{equation}\label{initial3}
w_{\rm tot}:=\sum_{i=1}^{n}{w_i},
\end{equation}
be the total number of sell and buy orders for all assets that
have been executed during the $m$ time instants prior to
$[t_0,t_0+T]$ (we will set $t_0:=0$).

Let also
\begin{equation}\label{asprd}
\bar{spr}_i:=\displaystyle{\sum_{j=1}^m}\frac{A_i(t_j)-
B_i(t_j)}{m},
\end{equation}
be the average spread of $i$ share; we define
\begin{equation}\label{initial4}
\alpha_i=\frac{w_i}{\bar{spr}_i},
\end{equation}
as the measure of liquidation of asset $i$.

The liquidity coefficient $\alpha_{\rm in}$ is defined then as
follows
\begin{equation}\label{initial5}
\alpha_{\rm in}:=\sum_{i=1}^{n}{a_i\frac{w_i}{w_{\rm
tot}}}=\sum_{i=1}^{n}{\frac{w_i^2}{\bar{spr}_iw_{\rm tot}}}.
\end{equation}

However, since we will use a logarithmic scale for the space
variables, we shall define in the rescaled problem \eqref{stef}
or \eqref{stef2} $\alpha$ as follows
\begin{equation}\label{initial5*}
\alpha:=\sum_{i=1}^{n}\frac{w_i^2}{(lspra_i)w_{\rm tot}}.
\end{equation}

\begin{remark}
The formulae \eqref{initial3}-\eqref{initial5*} were implemented
for the case of one market participating and so for only one ball
for the initial solid phase of zero trading. In the general case
of $I$ balls (see \eqref{stef2}), we apply the same formulae for
the limit order books of each $i=1,\cdots,I$ market for the same
$n$ assets, and compute each respective liquidity coefficient;
let this be denoted by $\alpha^i$. We may then define as $\alpha$
the average, i.e.
$$\alpha:=\frac{1}{I}\displaystyle{\sum_{i=1}^I}\alpha^i.$$
\end{remark}

\subsubsection{An example}\label{finex} We consider a portfolio of three
assets ($n=3$) and historical data for $m=5$ time instances.
Tables \ref{tab:1}, \ref{tab:2}, \ref{tab:3} show the data of the
respective order books.
\begin{table} [ht]
\caption{A sample of 5 quotes for asset 1}
\label{tab:1}
\begin{tabular}{lcccc}
\hline\noalign{\smallskip}
Time $t_j$& $A_1(t_j)$ & $B_1(t_j)$ & $spr_1(t_j)$ & $\frac{A_1(t_j)+B_1(t_j)}{2}$ \\
\noalign{\smallskip}\hline\noalign{\smallskip}
9:00& 30.25&29.75 & 0.5&30 \\
9:02& 30.75&29.50&1.25& 30.125 \\
9:04& 31.00&29.25&1.75&30.125  \\
9:06& 31.50&29.00 & 2.50&30.25\\
9:08& 35.00&28.75 &6.25&31.875\\
\noalign{\smallskip}\hline\noalign{\smallskip}
Sum&158.5&146.25&12.25 &152.375\\
$\bar{spr}_1$&&$12.25/5=2.45$\\
$lspra_1$&&$\ln(158.5)-\ln(146.25)=0.080437107$\\
$x_{c1}$&& $\ln(152.375/5)=3.416906675$\\

 \noalign{\smallskip}\hline
\end{tabular}
\end{table}
\begin{table} [ht]
\caption{A sample of 5 quotes for asset 2}
\label{tab:2}
\begin{tabular}{lcccc}
\hline\noalign{\smallskip}
Time $t_j$& $A_2(t_j)$ & $B_2(t_j)$ & $spr_2(t_j)$ & $\frac{A_2(t_j)+B_2(t_j)}{2}$ \\
\noalign{\smallskip}\hline\noalign{\smallskip}
9:00& 15.00&14.25 &0.75 &14.625 \\
9:02& 15.25&14.25 &1.00 &14.75 \\
9:04& 15.25& 15.00&0.25 &15.125 \\
9:06& 15.50&15.25 &0.25 &15.375\\
9:08& 15.75& 15.50&0.25 &15.625\\
\noalign{\smallskip}\hline\noalign{\smallskip}
Sum&76.75&74.25&2.50 &75.50\\
$\bar{spr}_2$&&$2.50/5=0.5$\\
$lspra_2$&&$\ln(76.75)-\ln(74.25)=0.033115609$\\
$x_{c2}$&& $\ln(75.50/5)=2.714694744$\\
\noalign{\smallskip}\hline
\end{tabular}
\end{table}
\begin{table} [ht]
\caption{A sample of 5 quotes for asset 3}
\label{tab:3}
\begin{tabular}{lcccc}
\hline\noalign{\smallskip}
Time $t_j$& $A_3(t_j)$ & $B_3(t_j)$ & $spr_3(t_j)$ & $\frac{A_3(t_j)+B_3(t_j)}{2}$ \\
\noalign{\smallskip}\hline\noalign{\smallskip}
9:00& 20.75&19.50 &1.25 &20.125 \\
9:02&21.00 &19.50 &1.50 &20.25 \\
9:04& 21.25&19.25 &2.00 &20.25  \\
9:06& 22.00&18.25 &3.75 &20.125 \\
9:08& 25.50& 18.50&7.00 &22.00\\
\noalign{\smallskip}\hline\noalign{\smallskip}
Sum&110.5&95&15.50 & 102.75\\
$\bar{spr}_3$&&$15.50/5=3.1$\\
$lspra_3$&&$\ln(110.5)-\ln(95)=0.151138629$\\
$x_{c3}$&& $\ln(102.75/5)=3.022860941$\\
\noalign{\smallskip}\hline
\end{tabular}
\end{table}
Table \ref{tab:4} shows the number of shares of the assets sold
and bought in the $m=5$  periods,
\begin{table} [ht]
\caption{Number of shares sold, and liquidity coefficient}
\label{tab:4}
\begin{tabular}{lcccc}
\hline\noalign{\smallskip}
Asset & $w_i$ &$a_i=w_i/\bar{spr}_i$ &$w_i/w_{\rm tot}$ &$a_i w_i/w_{\rm tot}$\\
\noalign{\smallskip}\hline\noalign{\smallskip}
1& 550 &550/2.45=224.4897959&550/1600=0.34375&77.16836735\\
2& 750 &750/0.5=1500&750/1600=0.46875&703.125\\
3& 300 &300/3.1=96.77419355&300/1600=0.1875&18.14516129\\
\noalign{\smallskip}\hline\noalign{\smallskip}
Sum& 1600& & &$\alpha_{\rm in}=798.4385286$\\
\noalign{\smallskip}\hline
\end{tabular}
\end{table}
or in the logarithmic scale, for defining $\alpha$ we use the
Table \ref{tab:5}.
\begin{table} [ht]
\caption{Number of shares sold, and liquidity coefficient in
logarithmic scale} \label{tab:5}
\begin{tabular}{lcccc}
\hline\noalign{\smallskip}
Asset & $w_i$ &$w_i/lspr_i$ &$w_i/w_{\rm tot}$ &$( w_i/lspra_i)(w_i/w_{\rm tot})$\\
\noalign{\smallskip}\hline\noalign{\smallskip}
1& 550 &550/0.080437107=6837.640227&550/1600=0.34375&2350.438828\\
2& 750 &750/0.033115609=22647.93031&750/1600=0.46875&10616.21733\\
3& 300 &300/0.151138629=1984.932649&300/1600=0.1875&372.1748718\\
\noalign{\smallskip}\hline\noalign{\smallskip}
Sum& 1600& & &$\alpha=13338.83103$\\
\noalign{\smallskip}\hline
\end{tabular}
\end{table}
Based on the above data, we derive the following parameters: the
initial center of the spherical domain at time $0$ is given by
\begin{equation}\label{nex}
x_c=(x_{c1},x_{c2},x_{c3})=(3.416906675,2.714694744,3.022860941),
\end{equation}
the radius of the spherical domain at time $0$ by
\begin{equation}\label{ner}
R(0)=\frac{1}{2}\min\{0.080437107,0.033115609,0.151138629\}=0.016557805,
\end{equation}
and the coefficient $\alpha_{\rm in}$ is
 $$ \alpha_{\rm in}=798.4385286,$$
 while at the logarithmic scale
 \begin{equation}\label{nea}
 \alpha=13338.83103.
 \end{equation}

\begin{remark}
At a next section we will use this computed value of $\alpha$,
given by \eqref{nea}, and the specific data presented as above
(together with the values in \eqref{nex}, \eqref{ner}) in a
simulation where the stochastic Stefan problem \eqref{stef2} with
one initial ball in the zero trading area will be solved
numerically for the corresponding time of $8$ minutes in the
financial day; the data used refer to the number of shares traded
(sold and bought) in $8$ minutes and the liquidity coefficient
numerator uses this number which is highly increasing during the
day, while the denominator involves the spread that tends to be
less varying.
\end{remark}
\section{Asymptotic expansions and approximating dynamics in dimensions $n=3$}
\subsection{Preliminaries}
The deterministic version of the general stochastic Stefan problem
\eqref{stef} in the union of balls solid phase statement
\eqref{stef2}, i.e. when $n=3$ and $\sigma=0$, has been fully
analyzed in \cite{nietthes}. Our aim is to derive through
asymptotic expansions the approximating dynamics of the moving
boundary of \eqref{stef2} in the presence of noise (stochastic
volatility) as a system of stochastic differential equations.
This will provide a useful tool for the prediction of the spreads
of $3$ shares participating in $I$ markets since the system can
be solved numerically. In particular, for various cases of
financial interest we will present the numerical results of a
number of simulations. We note that the analysis here is
restricted to $n=3$ (as in \cite{nietthes}) but can be easily
extended for $n\neq 3$, once careful calculations are applied in
the derivation of the statement of \eqref{stef2} in dimensions
$n\neq 3$; the surface area of a ball is present at the Stefan
condition and will involve $n$, while the $n$-dependent euclidean
norm in $\mathbb{R}^n$ will appear and may modify many other
formulae.

\subsection{Zero volatility}
First we analyze briefly the known results of \cite{nietthes} in
the absence of noise, and then we solve the approximating ODEs
system numerically; this numerical part appears for first time in
the literature.

We present first some existing results for the problem
\eqref{stef2}, for $\sigma=0$.

The deterministic Stefan problem \eqref{stef2}, takes the form
\begin{equation}\label{stef3}
\begin{split}
\alpha^{-1}\partial_t v=&\Delta
v,\;\;x\in\mathbb{R}^3-\mathcal{D}(t)\;\;{\rm('liquid'\;
phase)},\;\;t>0,\\
v=&0,\;\;x\in\mathcal{D}(t)\;\;{\rm('solid'\;phase)},\\
v=&\frac{1}{R_i(t)}\;\;{\rm on}\;\;\Gamma_i(t),\\
\dot{R}_i(t)=&\frac{1}{4\pi R_i^2(t)}\int_{\Gamma_i(t)}\nabla v\cdot \eta,\\
\Gamma(0)&=\Gamma_0,\\
v_\infty(t)&=\displaystyle{\lim_{r\rightarrow\infty} v(r,t)},
\end{split}
\end{equation}
where the initial $\Gamma_0$ is given.

The so-called mean-field variable $v_\infty$ describes the
limiting behaviour of the density $v$ away from the phase
transitions interfaces. A scaling on the space variables of the
form $x\in[0,1]\rightarrow\delta^{-4} x\in[0,\delta^{-4}]$, where
$\alpha^{-1}=\delta^{9}$, cf. \cite{nietthes,JDE12}, approximates
as $\alpha\rightarrow \infty$, the background domain
$\mathbb{R}^3$ of the moving boundary problem by some domain
$\Omega$ of very large volume
$|\Omega|=(\delta^{-4})^3=\delta^{-12}=\alpha^{4/3}.$ The
assumption
$${\rm Vol}(\rm solid\;phase)<<{\rm Vol}(\Omega)=:|\Omega|,$$ leads to
$$\displaystyle{\sum_{i=1}^I} R_i(0)\leq I \displaystyle{\max_i}
R_i(0)<<{\rm
diameter}(\Omega)=\mathcal{O}(\delta^{-4})=\mathcal{O}(\alpha^{4/9}).$$
So, we impose the next condition for our initial data
\begin{equation}\label{volfr}
I\displaystyle{\max_{i=1,\cdots,I}}R_i(0)<<\mathcal{O}(\alpha^{4/9}),
\end{equation}
which is a condition for the proper scaling of the problem
\eqref{stef2}.

In our approach we will not use a $\delta$ for rescaling the
equation (as done in \cite{nietthes,JDE12}, where $\delta<<1$ is
used also in relation with a very large number of radii, in a
macroscopic level, not needed here), but we will consider instead
the $\alpha$-dependent Stefan problem for $\alpha$ and $R_i(0)$
satisfying \eqref{volfr}.

It is known that as $\alpha^{-1}\rightarrow 0$, the exact solution
$$v_\infty(t)+\sum_i\frac{1-R_i(t)v_\infty(t)}{|x-x_c^i|},$$ of the
quasi-static elliptic problem (replace $0$ at the left-hand side
of the pde of \eqref{stef3}) for $x_c^i\in\mathbb{R}^3$ the
center of the ball $B_{R_i}$, approximates the solution $v$ of
\eqref{stef3}; see also the comments at pg. 4683 of \cite{JDE12}.

For $\alpha>0$ very large, which describes here a strong diffusion
of the sell/buy orders, $v_{\infty}$ satisfies approximately the
i.v.p.
\begin{equation}\label{1.5*}
\partial_t v_{\infty}(t)=4\pi\alpha^{-1/3}\displaystyle{\sum_{i\in
I}}(1-R_i(t)v_\infty(t)),\;\;\;\;\;v_\infty(0)=v_{\infty 0},
\end{equation}
for example for $v_{\infty
0}\approx\frac{I}{\displaystyle{\sum_{i\in I}} R_i(0)}$. Here,
$\alpha^{-1/3}$ corresponds to
$$\frac{1}{\alpha^{-1}|\Omega|}=\frac{1}{\alpha^{-1}a^{4/3}}=\alpha^{-1/3},$$
see at pg. 4684 of \cite{JDE12}, and in the sequel of this
section.

Also, in the weak sense, an approximating formula for the
dynamics of the radii as $\alpha\rightarrow \infty$ is given by
\begin{equation}\label{1.6}
\dot{R}_i(t)= \frac{v_{\infty}(t)}{R_i(t)}-\frac{1}{R_i^2(t)},
\end{equation}
for $v_\infty(t)$ the solution of \eqref{1.5*}. See for example
the approximation estimate in $W^{1,1}(0,T)$ of \cite{nietthes}
at pg. 175 of \cite{nietthes} (for
$\|z\|_{W^{1,1}(0,T)}:=\int_0^T (|z|+|z_t|)dt$), or at pg. 4712
of \cite{JDE12} for $\beta=g_i:=0$ in the formula (87) therein,
derived for the rescaled problem.

Also, the density solution $v$ is approximated by the
quasi-static one
\begin{equation}\label{apo}
v(x,t)\approx
v_\infty(t)+\sum_i\frac{1-R_i(t)v_\infty(t)}{|x-x_c^i|}.
\end{equation}

The ODEs system \eqref{1.5*}, \eqref{1.6} for the dynamics of the
$I$ radii consists of $I+1$ equations with unknowns
$$v_{\infty}(t),\;R_1(t),\;R_2(t),\cdots,R_I(t),$$
and initial values
\begin{equation}\label{iva}
v_{\infty}(0):=I\Big{(}\displaystyle{\sum_{i=1}^I}R_i(0)\Big{)}^{-1},\;R_1(0),\;R_2(0),\cdots,R_I(0).
\end{equation}
We rewrite the system in the equivalent form
\begin{equation}\label{syst1}
\begin{split}
&\partial_t
v_{\infty}(t)=4\pi\alpha^{-1/3}\displaystyle{\sum_{i\in
I}}(1-R_i(t)v_\infty(t)),\;\;\;\;\;v_\infty(0)=I\Big{(}\displaystyle{\sum_{i=1}^I}R_i(0)\Big{)}^{-1},\\
&\partial_t(R_i^3)(t)=
3v_{\infty}(t)R_i(t)-3,\;\;\;\;\;R_i(0):=R_{i0},\;\;i=1,\cdots,I,
\end{split}
\end{equation}
or by setting
 $$z_i(t):=R_i^3(t),$$ we obtain the equivalent system
\begin{equation}\label{ssss0}
\begin{split}
&\partial_t
v_{\infty}(t)=4\pi\alpha^{-1/3}\displaystyle{\sum_{i\in
I}}(1-z_i(t)^{1/3}v_\infty(t)),\;\;\;\;\;v_\infty(0)=I\Big{(}\displaystyle{\sum_{i=1}^I}R_i(0)\Big{)}^{-1},\\
&\partial_t z_i(t)=
3v_{\infty}(t)z_i(t)^{1/3}-3,\;\;\;\;\;z_i(0):=R_{i0}^{3},\;\;i=1,\cdots,I.\\
\end{split}
\end{equation}
Here, each equation for $z_i$ holds until the vanishing time of
 the $i$ ball.

The solution $z_i(t)$, $v_\infty(t)$ of the above
 system of ODEs is then used to specify $R_i(t)$ and $v(x,t)$ by
 \begin{equation}\label{ssss1}
 \begin{split}
&R_i(t)=z_i(t)^{1/3},\;\;i=1,\cdots,I,\\
&v(x,t)=v_\infty(t)+\sum_i\frac{1-R_i(t)v_\infty(t)}{|x-x_c^i|},\;\;i=1,\cdots,I,
\end{split}
\end{equation}
where for the equation for $v$ we used the approximation formula
given by \eqref{apo}.
\begin{remark}
The system \eqref{ssss0} will be solved numerically, and its
solution will be used in the direct formulae \eqref{ssss1}. We
propose \eqref{ssss0}, \eqref{ssss1} (and their numerical
solution) as a financial tool for the estimation at time $t$ of
the spreads $2R_i(t)$ and the density $v(x,t)$ of the Stefan
problem \eqref{stef2} when $\sigma=0$.
\end{remark}

\subsubsection{Numerical experiments}
We constructed a double precision Matlab code for the numerical
solution of the system \eqref{ssss0}, \eqref{ssss1}; there, we
used the ODE45 routine.

We applied our code for a number of numerical experiments with
initial data satisfying the proper scaling condition
\eqref{volfr}.
\begin{enumerate}
\item \underline{$4$ radii}:

For the first experiment, we took $I:=4$ balls for the initial
solid phase, and $R_i(0):=1,\;2,\;3,\;8$, and $\alpha=10000$. For
the graphs of the radii as functions of $t$, and their vanishing
times at the horizontal $t$- axis, see Fig. \ref{fig1}.
Obviously, the expected dominance of the larger ball at the
expense of the smaller ones is observed.
\begin{figure}[h]
\includegraphics[width=8cm]{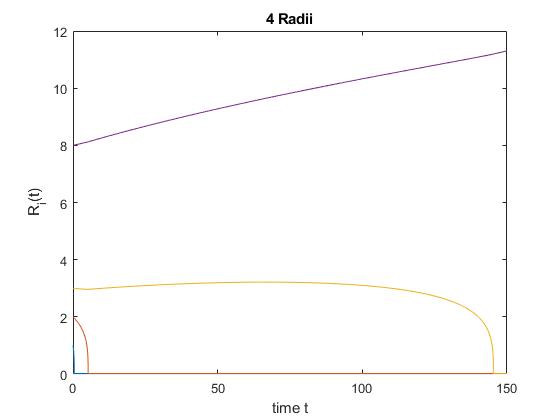}
\caption{Radii dynamics of $4$ balls at the solid
phase.}\label{fig1}
\end{figure}
\item \underline{$100$ radii}:

We checked our code for a very large number $I=100$ of initial
balls, with centers a small perturbation of $x_{\rm
intr}\in\mathbb{R}^3$ where an intrinsic value $\|x_{\rm
intr}\|=15$ is assigned. The initial radii are defined in a
comparative way through $x_{\rm intr}$ by $R_i(0):=\delta\% v_i,$
for $v_i:=\|p_i x_{\rm intr}\|$, $p_i\in[0,1]$ (randomly
evaluated), and $\delta:=25$. We took $\alpha=10000000$. In this
run the initial data are given by a random perturbation of a
historical values set of data; here, $I=100$ does not represent
$I$ different markets (as in the main financial application we
presented) and $R_i$ are not related to spreads. The next figure,
Fig. \ref{fig2} presents the evolution of the radii.
\begin{figure}[h]
\includegraphics[width=8cm]{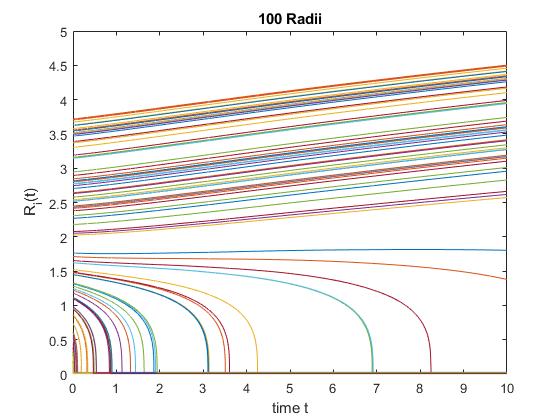}
\caption{Radii dynamics of $100$ balls at the solid
phase.}\label{fig2}
\end{figure}

\item \underline{$2$ radii}:

We took $\alpha=1000$, $R_1(0):=2.5,\;R_2(0):=1.5.$ We present
the dynamics of the $2$ radii at the next figure, Fig. \ref{fig3}.
\begin{figure}[h]
\includegraphics[width=8cm]{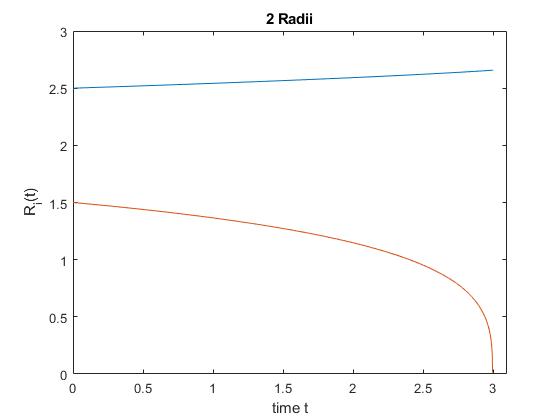}
\caption{Radii dynamics of $2$ balls at the solid
phase.}\label{fig3}
\end{figure}
\item \underline{$1$ ball at the solid phase with very large radius (large
spread case)}:

We took one sphere of center $x_1=(3/1000,7/1000,15/1000)$ and
initial radius $$R(0)=\|x_1\|250/100=4.205650960315179e-01,$$ and
defined $\alpha=100$; this case exceeds severely a normal
percentage between the spread and the value of the asset measured
by $\|x_1\|$.  Recall that in the financial application analyzed
in the previous sections the diameter $2R(0)$ stands as a measure
of the minimum spread of the $3$ shares at the initial time. Our
run demonstrated a sudden drop of the radius, see Figure
\ref{fig4}. Here, we remind that the scaling of initial data
satisfied \eqref{volfr}. However, one ball is a static solution
and theoretically it is expected its radius to change very
slowly, as seen at the next experiment.
\begin{figure}[h]
\includegraphics[width=8cm]{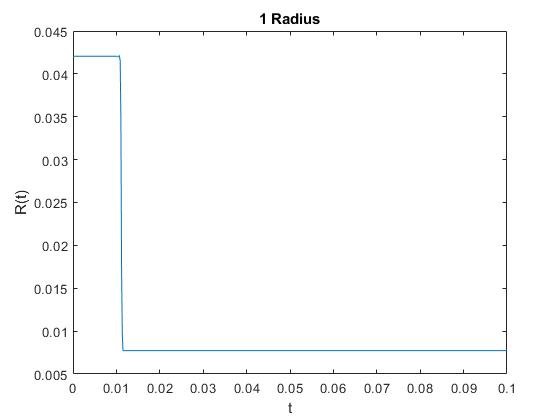}
\caption{Radius dynamics of one ball at the solid phase with
relatively large spread.} \label{fig4}
\end{figure}

\item \underline{$1$ ball at the solid phase with small radius (small
spread case)}:

Finally, we took
$$R(0):=\|x_1\|25/100=4.205650960315179e-02,$$ (more normal
range between value and spread), and we kept the same other data
as in the previous experiment; we derived numerically the
expected quasi-static solution approximate profile $R(t)\approx
R(0)$ for all $t$ (remind that one ball is an equlibrium of the
quasistatic case), see Figure \ref{fig5}.
\begin{figure}[h]
\includegraphics[width=8cm]{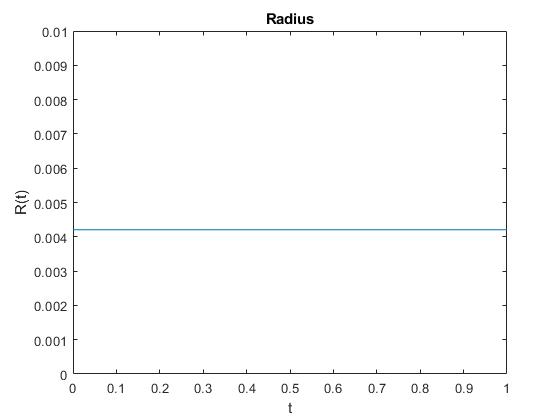}
\caption{Radius dynamics of one ball at the solid phase with
relatively small spread.} \label{fig5}
\end{figure}
\end{enumerate}
\subsection{Formal asymptotics for the stochastic Stefan problem with time noise} We proceed to the formal calculations analogous
to those presented in \cite{JDE12} (by defining the parameters
$\beta$, $g_i$ of the Stefan problem of \cite{JDE12} as
$\beta=g_i=0$) and additionally, we insert the extra noise term
in the parabolic equation.

First we present the result of Lemma \ref{itocor} for the formula
of differentiation in time of integrals on domains of stochastic
time dependent spherical boundary; its proof involves integrals
defined on stochastic on time spherical surfaces embedded in
$\mathbb{R}^3$ (case of stochastic radius); \eqref{formito} there
is used in the sequel for the second order asymptotics of the
problem's stochastic dynamics, in case of time noise given as the
formal derivative of a Wiener process; see at the Appendix for
the analytical proof, and for the version where
$\partial_t\int_{B(R(t))}u(x,t)dx$ is computed (Lemma
\ref{itocor2}) that we included for completeness of the text.

Let $x\in\mathbb{R}^3,\;t\in\mathbb{R}$, and $u=u(x,t)$, $R=R(t)$,
be real stochastic processes compatible with It\^o calculus in
time, and let $u$ be smooth in space. If $B_{R(t)}=:B(R(t))$ is a
ball in $\mathbb{R}^3$ of radius $R(t)$, the \eqref{formito}
holds i.e.
\begin{equation*}\label{formito}
\begin{split}
\partial_t\int_{\mathbb{R}^3-B(R(t))}u(x,t)dx=&\int_{\mathbb{R}^3-B(R(t))}u_t(x,t)dx
-\dot{R}(t)\Big{[}1+\frac{\dot{R}(t)}{R(t)}\Big{]}\int_{\partial
B(R(t))}u(s,t)ds\\
&-\frac{(\dot{R}(t))^2}{2}\int_{\partial B(R(t))}\nabla
u(s,t)\cdot \eta(s) ds-\dot{R}(t)\int_{\partial B(R(t))}u_t(s,t)
ds,
\end{split}
\end{equation*}
if the appearing integrals are well defined. Here,
$\dot{R}:=R_t=dR(t)$, and $\eta$ is the outward normal vector to
$\partial B(R)$.

\begin{remark}
In the deterministic case, due to the usual chain rule, in
dimensions $n=3$, and for general $u$, the result of Lemma
\ref{itocor2} takes the form
\begin{equation}\label{formito2d}
\begin{split}
\partial_t\int_{B(R(t))}u(x,t)dx=&\int_{B(R(t))}u_t(x,t)dx
+\dot{R}(t)\int_{\partial B(R(t))}u(s,t)ds,
\end{split}
\end{equation}
which is a well known formula.

Moreover in the deterministic case again, the result of Lemma
\eqref{itocor} takes the form
\begin{equation}\label{formitod}
\begin{split}
\partial_t\int_{\mathbb{R}^3-B(R(t))}u(x,t)dx=&\int_{\mathbb{R}^3-B(R(t))}u_t(x,t)dx
-\dot{R}(t)\int_{\partial B(R(t))}u(s,t)ds,
\end{split}
\end{equation}
for $\dot{R}:=R_t$.
\end{remark}

Let us consider the problem \eqref{stef2} posed in
$\mathbb{R}^3$, with non-smooth noise $\dot{W}:=\dot{W}(t)$,
depending only on time given as the formal derivative of a time
dependent one dimensional, one parameter Wiener process (for
example $W(t)=\beta(t)$ a brownian process). This problem for one
only (open) ball $B_R$ with radius $R$, has the following
statement
\begin{equation}\label{stef2one}
\begin{split}
\alpha^{-1}\partial_t v=&\Delta v +\alpha^{-1}\sigma({\rm
dist}(x,\partial R(t)))\dot{W}(t),\;\;x\in\mathbb{R}^3-B_R
(t),\;\;t>0,\\
v=&0,\;\;x\in B_R(t),\\
v=&\frac{1}{R(t)}\;\;{\rm on}\;\;\partial B_R(t),\\
\dot{R}(t)&=\frac{1}{4\pi R^2(t)}\int_{\partial B_R(t)}\nabla
v\cdot \eta,
\end{split}
\end{equation}
and $\displaystyle{\lim_{r\rightarrow
\infty}}v(r,t)=v_\infty(t),$ for $r$ the distance of $x\in
\mathbb{R}^3$ from the origin.

The formal construction of an approximate solution for the
multiple spheres problem \eqref{stef2}, is based on the following
argument. Near one of the spherical domains of the solid phase of
\eqref{stef2}, the solution of \eqref{stef2}, should look
approximately like the solution of the single spherical domain
solid phase problem \eqref{stef2one}; see the analogous argument
in \cite{JDE12}.

The quasi-static version of \eqref{stef2one} is given as
$\alpha^{-1}\rightarrow 0$ by
\begin{equation}\label{qstef2}
\begin{split}
\Delta v =&0,\;\;x\in\mathbb{R}^3-B_R
(t),\;\;t>0,\\
v=&0,\;\;x\in B_R(t),\\
v=&\frac{1}{R(t)}\;\;{\rm on}\;\;\partial B_R(t),\\
\dot{R}(t)&=\frac{1}{4\pi R^2(t)}\int_{\partial B_R(t)}\nabla
v\cdot \eta,
\end{split}
\end{equation}
with $\displaystyle{\lim_{r\rightarrow
\infty}}v(r,t)=v_\infty(t)$;
 observe that $\alpha^{-1}$ acts also to the noise term of \eqref{stef2one},
 which thus, vanishes in the quasi-static case.

The exact solution of \eqref{qstef2} is given by
\begin{equation}\label{ex1}
v(r,t)=v_\infty(t)+\frac{1-R(t)v_\infty(t)}{r},
\end{equation}
and
\begin{equation}\label{ex2}
\dot{R}(t)=\frac{v_\infty(t)}{R(t)}-\frac{1}{R^2(t)}.
\end{equation}

Since $\alpha^{-1}<<1$, the solution $v(x,t)$ of \eqref{stef2} for
time noise $\dot{W}:=\dot{W}(t)$, is approximated by a linear
combination of individual (single sphere) solutions of the
quasi-static problem \eqref{qstef2}, as follows
\begin{equation}\label{apsol1}
v(x,t)\approx
v_\infty(t)+\sum_i\frac{1-R_i(t)v_\infty(t)}{|x-x_c^i|},
\end{equation}
for $x_c^i$ the center of the ball $B_{R_i}$ with radius $R_i$.

As $\alpha^{-1}\rightarrow 0$ the background domain $\mathbb{R}^3$
of the moving boundary problem is approximated by some domain
$\Omega$ of very large volume $|\Omega|=\alpha^{4/3}$. Moreover,
the liquid phase $\mathbb{R}^3-\mathcal{D}$ is very close to
$\Omega$.

Hence, we consider that
$$|\Omega|v_\infty=\int_{\Omega}v_\infty dx\approx \int_{\mathbb{R}^3-\mathcal{D}}vdx,$$
which yields that
$$|\Omega|v_\infty\approx \int_{\mathbb{R}^3-\mathcal{D}}vdx.$$
We integrate in the liquid phase both sides of the stochastic
equation of \eqref{stef2}, use the above approximation, and the
b.c. of \eqref{stef2}, and derive for
$$m(t):=\displaystyle{\sum_i}\frac{(dR_i)^2}{2}\int_{\partial
B_{R_i}}\nabla v\cdot \eta
ds+\displaystyle{\sum_i}dR_i\int_{\partial B_{R_i}} v_t
ds=2\pi\displaystyle{\sum_i}(dR_i)^3R_i^2-4\pi
\sum_iR_i\dot{R_i}^2\frac{1}{R_i+\dot{R}_i},$$ where we used
It\^o calculus to differentiate $v=1/R_i$ on the spheres,
$$v_t=-\dot{R_i}/(R_i(R_i+\dot{R}_i)),\;\;{\rm on}\;\;\partial B_i,$$
\begin{equation}\label{fineq}
\begin{split}
|\Omega|\partial_t
v_\infty\approx&\partial_t\int_{\mathbb{R}^3-\mathcal{D}}vdx=\int_{\mathbb{R}^3-\mathcal{D}}v_tdx-\displaystyle{\sum_i}dR_i[1+dR_i
R_i^{-1}] \int_{\partial B_{R_i}}vds-
m(t)\\
=&\int_{\mathbb{R}^3-\mathcal{D}}\alpha \Delta v
dx+\dot{W}(t)\int_{\mathbb{R}^3-\mathcal{D}}\sigma({\rm
dist}(x,\Gamma(t)))dx-\displaystyle{\sum_i}\int_{\partial B_{R_i}}\dot{R}_i[1+\dot{R}_iR_i^{-1}]vds-m(t)\\
=&-\int_{\cup\partial B_{R_i}}\alpha \nabla v\cdot \eta ds
+\dot{W}(t)\int_{\mathbb{R}^3-\mathcal{D}}\sigma({\rm
dist}(x,\Gamma(t)))dx- \displaystyle{\sum_i}\int_{\partial B_{R_i}}\dot{R}_i[1+\dot{R}_iR_i^{-1}]vds-m(t)\\
=&-\alpha 4\pi\sum_i R_i^2\dot{R}_i
+\dot{W}(t)\int_{\mathbb{R}^3-\mathcal{D}}\sigma({\rm
dist}(x,\Gamma(t)))dx- \displaystyle{\sum_i}\int_{\partial
B_{R_i}}\dot{R}_i[1+\dot{R}_iR_i^{-1}]vds-m(t),
\end{split}
\end{equation}
where we used Lemma \ref{itocor} (formula \eqref{formito}). So,
we arrive at
\begin{equation}\label{fineq11}
\begin{split}
\partial_t v_\infty\approx &-\frac{\alpha}{|\Omega|}
4\pi\sum_i R_i^2\dot{R}_i
+\frac{1}{|\Omega|}\dot{W}(t)\int_{\mathbb{R}^3-\mathcal{D}}\sigma({\rm
dist}(x,\Gamma(t)))dx\\
&-\frac{1}{|\Omega|}\Big{[}\displaystyle{\sum_i}\int_{\partial
B_{R_i}}\dot{R}_i[1+\dot{R}_iR_i^{-1}]vds+m(t)\Big{]}.
\end{split}
\end{equation}

Using in the above that $\alpha>>1$, we ignore the last term.
However the same argument is avoided for the noise term (being
non smooth and not comparable). Replacing \eqref{ex2} for each
sphere, and using that $|\Omega|=\alpha^{4/3}$, we derive the
next system of stochastic differential equations for the
approximating dynamics of \eqref{stef2}
\begin{equation}\label{ex22}
\dot{R_i}(t)\approx\frac{v_\infty(t)}{R_i(t)}-\frac{1}{R_i^2(t)},\;\;\;\;i=1,\cdots,I,
\end{equation}
\begin{equation}\label{fineq22}
\begin{split}
\partial_t v_\infty(t)\approx 4\pi\alpha^{-1/3}
\displaystyle{\sum_{i=1}^I} (1-R_i(t)v_\infty(t))
+\alpha^{-4/3}\dot{W}(t)\int_{\mathbb{R}^3-\cup
B_{R_i}(t)}\sigma({\rm dist}(x,\cup\partial B_{R_i}(t)))dx,
\end{split}
\end{equation}
for $B_{R_i}(t)$ the balls of constant centers $x_c^i$ and radii
$R_i(t)$ respectively. Remind that the solution $v$ of the
stochastic Stefan (the density of the sell and buy orders in the
financial setting) is approximated by \eqref{apsol1}.

Note that for $\sigma=0$ \eqref{ex22}, \eqref{fineq22} coincide
to the rigorous first order asymptotics given by \eqref{ssss0},
\eqref{ssss1}.

\begin{remark}
For more general data, we do not replace $|\Omega|$, and also keep
the second order approximation term in \eqref{fineq11}, i.e. we
do not ignore
\begin{equation}\label{imp1}
\begin{split}
\frac{1}{|\Omega|}&\Big{[}\displaystyle{\sum_i}\int_{\partial
B_{R_i}}\dot{R}_i[1+\dot{R}_iR_i^{-1}]vds+m(t)\Big{]}\\
=&\frac{1}{|\Omega|}\Big{[}\displaystyle{\sum_i}\int_{\partial
B_{R_i}}\dot{R}_i[1+\dot{R}_iR_i^{-1}]\frac{1}{R_i}ds+2\pi\displaystyle{\sum_i}(dR_i)^3
R_i^2-4\pi\sum_i
R_i\dot{R_i}^2\frac{1}{R_i+\dot{R}_i}\Big{]}\\
=&\frac{1}{|\Omega|}\Big{[}4\pi \sum_i
R_i^2\dot{R}_i[1+\dot{R}_iR_i^{-1}]\frac{1}{R_i}+2\pi\displaystyle{\sum_i}(dR_i)^3
R_i^2-4\pi\sum_i
R_i\dot{R_i}^2\frac{1}{R_i+\dot{R}_i}\Big{]}\\
=&4\pi\frac{1}{|\Omega|} \sum_i
(R_i\dot{R}_i[1+\dot{R}_iR_i^{-1}]+(dR_i)^3 R_i^2/2-
R_i\dot{R_i}^2\frac{1}{R_i+\dot{R}_i})\\
\approx& 4\pi\frac{1}{|\Omega|} \sum_i\Big{(}
R_i[1+\dot{R}_iR_i^{-1}]\Big{[}\frac{v_\infty(t)}{R_i(t)}-\frac{1}{R_i^2(t)}\Big{]}+\frac{1}{2}R_i^2\Big{[}\frac{v_\infty(t)}{R_i(t)}-\frac{1}{R_i^2(t)}\Big{]}^3\Big{)}\\
&-4\pi\frac{1}{|\Omega|}\sum_i\frac{1}{R_i(R_i+\dot{R_i})}\Big{(}v_\infty(t)-\frac{1}{R_i(t)}\Big{)}^2\\
=&4\pi\frac{1}{|\Omega|} \sum_i
[1+\dot{R}_iR_i^{-1}]\Big{(}v_\infty(t)-\frac{1}{R_i(t)}\Big{)}+2\pi\frac{1}{|\Omega|}
\sum_i\frac{1}{R_i}\Big{(}v_\infty(t)-\frac{1}{R_i(t)}\Big{)}^3\\
&-4\pi\frac{1}{|\Omega|}\sum_i\frac{1}{R_i(R_i+\dot{R_i})}\Big{(}v_\infty(t)-\frac{1}{R_i(t)}\Big{)}^2,
\end{split}
\end{equation}
and derive the next formula (by replacing once again $\dot{R}_i$
by \eqref{ex22} in the above)
\begin{equation}\label{fineq22cor}
\begin{split}
\partial_t v_\infty(t)\approx &4\pi\frac{\alpha}{|\Omega|} \displaystyle{\sum_{i=1}^I}
(1-R_i(t)v_\infty(t))\\
&-4\pi\frac{1}{|\Omega|} \sum_i
\Big{(}v_\infty(t)-\frac{1}{R_i(t)}\Big{)}-4\pi\frac{1}{|\Omega|}
\sum_i
\Big{(}v_\infty(t)-\frac{1}{R_i(t)}\Big{)}^2\\
&-2\pi\frac{1}{|\Omega|}
\sum_i\frac{1}{R_i}\Big{(}v_\infty(t)-\frac{1}{R_i(t)}\Big{)}^3
\\
&+4\pi\frac{1}{|\Omega|}\sum_i\frac{1}{v_\infty(t)-\frac{1}{R_i}+R_i^2}\Big{(}v_\infty(t)-\frac{1}{R_i(t)}\Big{)}^2\\
&+\frac{1}{|\Omega|}\dot{W}(t)\int_{\mathbb{R}^3-\cup
B_{R_i}(t)}\sigma({\rm dist}(x,\cup\partial B_{R_i}(t)))dx,
\end{split}
\end{equation}
in place of \eqref{fineq22}.
\end{remark}

For this case we may use a general scaling for defining our domain
$|\Omega|$ approximating $\mathbb{R}^3$, of the form
$x\in[0,1]\rightarrow c_s x\in[0,c_s]$, and $|\Omega|=c_s^{3}$
for $c_s>>1$. However, we also consider $\alpha$ relatively large
(since the main argument was to approximate with the static
problem formula for the derivatives of the radii).

Thus, \eqref{fineq22cor} takes the form
\begin{equation}\label{fineq33cor}
\begin{split}
\partial_t v_\infty(t)\approx &4\pi\frac{\alpha}{c_s^3} \displaystyle{\sum_{i=1}^I}
(1-R_i(t)v_\infty(t))\\
&-4\pi\frac{1}{c_s^3} \sum_i
\Big{(}v_\infty(t)-\frac{1}{R_i(t)}\Big{)}-4\pi\frac{1}{c_s^3}
\sum_i
\Big{(}v_\infty(t)-\frac{1}{R_i(t)}\Big{)}^2\\
&-2\pi\frac{1}{c_s^3}
\sum_i\frac{1}{R_i}\Big{(}v_\infty(t)-\frac{1}{R_i(t)}\Big{)}^3
\\
&+4\pi\frac{1}{c_s^3}\sum_i\frac{1}{v_\infty(t)-\frac{1}{R_i}+R_i^2}\Big{(}v_\infty(t)-\frac{1}{R_i(t)}\Big{)}^2\\
&+\frac{1}{c_s^3}\dot{W}(t)\int_{\mathbb{R}^3-\cup
B_{R_i}(t)}\sigma({\rm dist}(x,\cup\partial B_{R_i}(t)))dx.
\end{split}
\end{equation}
We may chose for example $c_s:>>1$, and treat the problem of more
general diffusion constant $\alpha$, independent from the initial
radii order $\mathcal{O}(R_i(0))$ or the centers; the $c_s$ does
not depend on $\alpha$, and can take care of the initial radii
order and of the prices-centers (the initial price vector must
belong to $\Omega$ and the diameter of $\Omega$ is equal to
$\mathcal{O}(c_s)$).

Once the Stefan problem is used in a financial setting, we would
like to consider diffusion coefficients $\alpha$ with magnitude
not depending from the initial radii, or the placement of the
centers. Hence, we propose the second order approximation formula
\eqref{fineq33cor} instead of the first order one \eqref{fineq22}.

\begin{remark}
Since the liquid phase approximates the background domain of the
Stefan problem, we may consider
$$\int_{\mathbb{R}^3-\cup B_{R_i}(t)}\sigma({\rm
dist}(x,\cup\partial
B_{R_i}(t)))dx\approx\int_{\mathbb{R}^+}\sigma(r)dr,$$ for
$\sigma(r)$ a sufficiently decaying
 function as $r\rightarrow\infty$, or of compact support (satisfying for example $\sigma(r)\sim
\mathcal{O}(r^{-(1+a)})$, for some $a>0$, as
$r\rightarrow\infty$); see also in \cite{main1} the discussion of
analogous properties for $\sigma$ for a Stefan problem posed in
dimension $1$.
\end{remark}
\subsection{Numerical experiments}
The first set of numerical experiments considers one initial ball
for the solid phase and implements numerically the first and
second order approximation stochastic differential systems
proposed. For all cases we used a double precision Matlab code
and the ODE45 routine. All initial data satisfy the proper
scaling condition \eqref{volfr}.

\subsubsection{$1$ radius, first order versus second order
asymptotics with stochastic volatility} We took $\alpha=100$ and
$R(0)=1$.

We solved numerically the first order approximation stochastic
dynamics system \eqref{ex22}, \eqref{fineq22}, with initial
condition given by \eqref{iva}, for
$c_0=\displaystyle{\int_{\mathbb{R}^+}}\sigma(r)dr=1$, and
$W(t):=\beta(t)$ the brownian motion following the normal
distribution $N(0,t)$. The noise $dW(t)$ was approximated by
using the brownian increments as follows (finite differences)
\begin{equation}\label{incr}
dW(t)\simeq \frac{\beta(t_j)-\beta(t_{j-1})}{t_j-t_{j-1}},
\end{equation}
 for
$\mathcal{O}(t_j-t_{j-1})=10^{-6}$, $j=1,\cdots,J$, where
$0=t_0<t_1<\cdots<t_J:=15$, was the time discretization of our
numerical scheme. We applied a Monte Carlo simulation for $100$
realizations (100 runs) and computed, for each realization, the
radius $R(t)$ for $t\in[0,15]$, see figure \ref{fig8}. Moreover,
we plot the value $R(t)$ for $t=15$, for each realization, see
figure \ref{fig9}. Remind that as predicted by the theory of the
deterministic parabolic Stefan problem, the case of one spherical
initial solid phase boundary has an almost constant radius
profile in time $R(t)\simeq R(0)$($=1$ there), since the constant
sphere is a solution of the quasi-static deterministic problem.
The main observation of our experiment for the stochastic case is
in contrast to the previous property. There existed quite a few
realizations where the radius vanished at finite time $t<15$,
while in other the profile was oscillating. The computed
experimental mean value of $R(t)$, for $t=15$, was equal to
$0.6096590298805101$, and thus, significantly smaller than the
initial radius $R(0)=1$.

\begin{figure}[h]
\includegraphics[width=8cm]{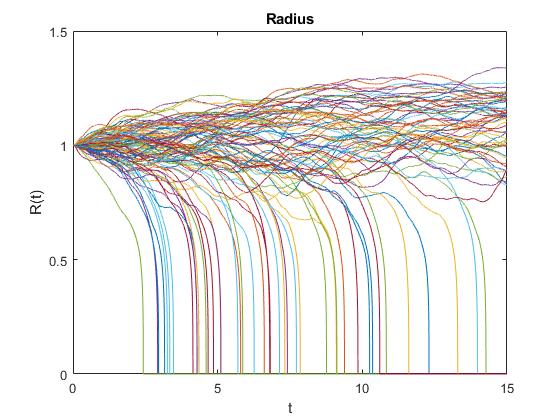}
\caption{100 realizations of $R(t)$, for $t\in[0,15]$, with first
order approximation.}\label{fig8}
\end{figure}
\begin{figure}[h]
\includegraphics[width=8cm]{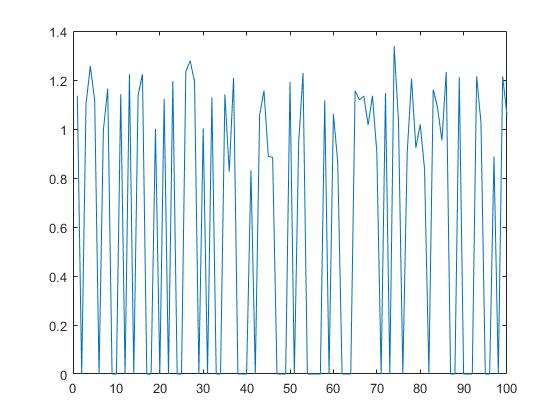}
\caption{100 realizations of $R(t)$, for $t=15$ (first order
approximation).}\label{fig9}
\end{figure}

We repeated the same experiment by using the second order
approximation for the stochastic dynamics system, \eqref{ex22},
\eqref{fineq33cor}, and \eqref{iva}, for $c_s^3=\alpha^{4/3}$. We
computed, for each realization, the radius $R(t)$ for
$t\in[0,15]$, see now figure \ref{fig8n}. We also plot the value
$R(t)$ for $t=15$, for each realization, see figure \ref{fig9n}.
Again there existed quite a few realizations where the radius
vanished at finite time $t<15$. The computed experimental mean
value of $R(t)$, for $t=15$, was equal to $0.69573807862059131$,
again, smaller than the initial radius $R(0)=1$. However, through
the $100$ realizations, the profile of $R(t)$ for $t=15$ was less
oscillating than this of the first order approximation, see fig.
\ref{fig9}, \ref{fig9n}.

\begin{figure}[h]
\includegraphics[width=8cm]{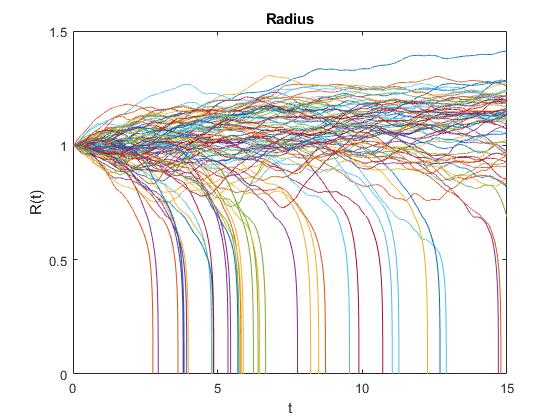}
\caption{100 realizations of $R(t)$, for $t\in[0,15]$, with second
order approximation.}\label{fig8n}
\end{figure}
\begin{figure}[h]
\includegraphics[width=8cm]{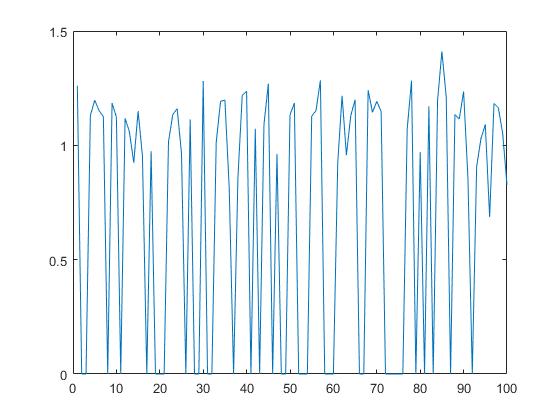}
\caption{100 realizations of $R(t)$, for $t=15$ (second order
approximation).}\label{fig9n}
\end{figure}

\subsubsection{Financial data experiment}
The next set of runs was devoted to the financial application
presented in Section \ref{finex} and the tables therein; we also
used the computed values given by \eqref{nex}, \eqref{ner},
\eqref{nea}. We considered one initial ball of center
$(3.416906675,2.714694744,3.022860941),$ and radius
$R(0)=0.016557805$, while the liquidity coefficient was given by
$\alpha=13338.83103$. Note that the above financial data happen
to satisfy the scaling condition \eqref{volfr}. We applied our
double precision Matlab code and implemented numerically the
second order approximation given by \eqref{ex22},
\eqref{fineq33cor}, with \eqref{iva}. For all runs the time length
used for the experiments was crucial and related to the computed
value of $\alpha$ by data given during $8$ consecutive minutes in
a financial day, cf. the tables in Section \ref{finex}.

The max ask price appeared in tables was  equal to $35.00$, i.e.
equal to $$\ln(35.00)=3.555348061489414e+00$$ in the logarithmic
scale. We took $c_s$ such that
\begin{equation*}
\begin{split}
c_s^3=|\Omega|>>&(4/3)\pi\ln(35.00)^3\\
=&\mbox{the volume of the ball of radius
}\ln(35.00)\\
=&1.882499980769812e+02,
\end{split}
\end{equation*}
 i.e. $c_s>>5.731192468848986e+00$ (we
note that in this experiment the radius is very small while the
max ask price in logarithmic scale is very larger, so since the
vectorial price is in $|\Omega|$ for the financial example, the
order of the measure of the vector price should be used instead
of the radius for the scaling). We took again
$c_0=\displaystyle{\int_{\mathbb{R}^+}}\sigma(r)dr=1$, and
$dW(t)$ was approximated by \eqref{incr}.

We used $c_s=5\times 10^4\times 5.731192468848986e+00$, and run
our Monte Carlo simulation for $300$ realizations in a time period
less or equal to $8$ minutes. In the first $2$ minutes the spread
(radius) had a very small increase while at the end of the $8$
minutes period the value (oscillating) was
$1.932132649058731e-02$, as shown in the next table.
$$\begin{tabular}{|c|c|}
  1.655780500000000e-02  &   0.000000000000000e+00 \;\;\mbox{minutes} \\ \hline
  1.655780500000002e-02  &   2.000000000000000e+00 \;\;\mbox{minutes}  \\ \hline
  2.028801437719118e-02  &   7.079646017699115e+00 \;\;\mbox{minutes} \\ \hline
  1.932132649058731e-02  &   8.000000000000000e+00 \;\;\mbox{minutes}  \\ 
\end{tabular}$$

When we used $2$ initial balls of different radii (but near the
radius of the previous example) under the same other data as
above, we observed the fast decrease of the smaller one, while
the larger was increasing.
\section{Appendix}
In this Appendix we present some important results of It\^o
calculus for space integrals on domains of stochastic boundary.
\begin{lemma}\label{itocor2}
Let $x\in\mathbb{R}^3,\;t\in\mathbb{R}$, and $u=u(x,t)$, $R=R(t)$,
be real stochastic processes compatible with It\^o calculus in
time, and let $u$ be smooth in space. Then for
$B_{R(t)}=:B(R(t))$ a ball in $\mathbb{R}^3$ of radius $R(t)$, it
holds that
\begin{equation}\label{formito2}
\begin{split}
\partial_t\int_{B(R(t))}u(x,t)dx=&\int_{B(R(t))}u_t(x,t)dx
+\dot{R}(t)\Big{[}1+\frac{\dot{R}(t)}{R(t)}\Big{]}\int_{\partial
B(R(t))}u(s,t)ds\\
&+\frac{(\dot{R}(t))^2}{2}\int_{\partial B(R(t))}\nabla
u(s,t)\cdot \eta(s) ds+\dot{R}(t)\int_{\partial B(R(t))}u_t(s,t)
ds,
\end{split}
\end{equation}
for $\dot{R}:=R_t=dR$ and $\eta$ the outward normal vector to
$\partial B(R)$.
\end{lemma}
\begin{proof}
Set
$$g(y,t):=\int_{B(y)}u(x,t)dx.$$

We aim to compute $\partial_t(g(R(t),t))$, i.e.
$$\partial_t\int_{B(R(t))}u(x,t)dx.$$

It\^o formula (2 variables Taylor) when $y$, $t$ depend
stochasticly, while $z$, $t$ do not depend stochasticly yields
\begin{equation}\label{iitoi}
\partial_t
g(y,z)=g_y\frac{y_t}{1!}+g_{yy}\frac{(y_t)^2}{2!}+g_{z}\frac{z_t}{1!}+g_{yz}\frac{y_tz_t}{1!1!}.
\end{equation}

In our case, $R(t)$ is stochastic, and $t$, $R(t)$ stochasticly
dependent, while $t$, $t$ are not depending stochasticly. So, by
applying \eqref{iitoi}, for $y:=R(t)$, $z=t$ and $y_t=\dot{R}$,
$z_t=1$, we obtain
\begin{equation}\label{l1}
\partial_t(g(R(t),t))=\dot{R}(t)g_y(R(t),t)+\frac{(\dot{R}(t))^2}{2}g_{yy}(R(t),t)+g_t(R(t),t)+g_{ty}(R(t),t)\dot{R}(t),
\end{equation}
i.e. for $d$ denoting the differentiation in $t$
$$d(g(R(t),t))=dR(t)g_y(R(t),t)+\frac{(dR(t))^2}{2}g_{yy}(R(t),t)+dg(R(t),t)+dg_{y}(R(t),t)dR(t).$$

 Considering the functional formula of $g$ as a function
$g:\mathbb{R}\times\mathbb{R}\rightarrow\mathbb{R}$, we will
compute $g_y$, $g_{yy}$, $g_t$ and  $g_{ty}$.

Moreover, we remind that our processes are smooth in space
variables and differentiation in space follows the usual calculus
(not It\^o).

By using spherical coordinates, we have
$$g(y,t)=\int_0^y\int_0^{2\pi}\int_0^\pi
\tau^2\hat{u}(\tau,\theta,\phi,t)\sin(\theta)d\theta d\phi
d\tau.$$ So, we obtain
$$g_y(y,t)=\int_0^{2\pi}\int_0^\pi
y^2\hat{u}(y,\theta,\phi,t)\sin(\theta)d\theta
d\phi=\int_{\partial B(y)}u(s,t)ds,$$ while
\begin{equation*}
\begin{split}
g_{yy}(y,t)=&2y\int_0^{2\pi}\int_0^\pi
\hat{u}(y,\theta,\phi,t)\sin(\theta)d\theta
d\phi+\int_0^{2\pi}\int_0^\pi
y^2\partial_y[\hat{u}(y,\theta,\phi,t)]\sin(\theta)d\theta
d\phi\\
=&\frac{2}{y}\int_{\partial{B}(y)} u(s,t)ds+\int_{\partial{B}(y)}
\nabla u(s,t)\cdot\eta(s)ds.
\end{split}
\end{equation*}
Moreover, we have
$$g_t(y,t)=\int_{B(y)}u_t(x,t)dx,$$
and differentiation in $t$ of $g_y$ gives
$$g_{yt}(y,t)=\int_{\partial
B(y)}u_t(s,t)ds.$$

 Replacing in \eqref{l1} we derive the result.
\end{proof}

The next Lemma is a direct result.

\begin{lemma}\label{itocor}
Let $x\in\mathbb{R}^3,\;t\in\mathbb{R}$, and $u=u(x,t)$, $R=R(t)$,
be real stochastic processes compatible with It\^o calculus in
time, and let $u$ be smooth in space. If $B_{R(t)}=:B(R(t))$ is a
ball in $\mathbb{R}^3$ of radius $R(t)$, it holds that
\begin{equation}\label{formito}
\begin{split}
\partial_t\int_{\mathbb{R}^3-B(R(t))}u(x,t)dx=&\int_{\mathbb{R}^3-B(R(t))}u_t(x,t)dx
-\dot{R}(t)\Big{[}1+\frac{\dot{R}(t)}{R(t)}\Big{]}\int_{\partial
B(R(t))}u(s,t)ds\\
&-\frac{(\dot{R}(t))^2}{2}\int_{\partial B(R(t))}\nabla
u(s,t)\cdot \eta(s) ds-\dot{R}(t)\int_{\partial B(R(t))}u_t(s,t)
ds,
\end{split}
\end{equation}
if the appearing integrals are well defined. Here,
$\dot{R}:=R_t=dR(t)$, and $\eta$ is the outward normal vector to
$\partial B(R)$.
\end{lemma}

\end{document}